\documentclass{article}
\usepackage[utf8]{inputenc}

\usepackage{amsfonts,amssymb,amsmath,amsthm}
\usepackage{bbm}
\usepackage{bbding}
\usepackage[english]{babel}
\usepackage[T1]{fontenc}
\usepackage[a4paper, total={6.7in, 10in}]{geometry}
\usepackage{calrsfs} 
\usepackage[dvipsnames,table]{xcolor}
\usepackage{upgreek}
\usepackage{enumerate}
\usepackage{enumitem}
\usepackage{stmaryrd}
\usepackage{graphicx}
\graphicspath{ {./images/} }
\usepackage{bbold}
\usepackage{mathrsfs}
\usepackage{tikz}
\usepackage{caption}
\usepackage[colorlinks]{hyperref}
\hypersetup{
    colorlinks=true,
    linkcolor=blue,
    filecolor=true,      
    urlcolor=blue,
    citecolor=cyan,
    }
    \usepackage{listings}
\usepackage{algorithm}
\usepackage[noend]{algpseudocode}
\algrenewcommand\algorithmicrequire{\textbf{Input:}}
\algrenewcommand\algorithmicensure{\textbf{Output:}}

\usepackage{subcaption} 

\newcommand{\C}{\mathcal{C}}

\newcommand{\K}{\mathbb{K}}
\newcommand{\Q}{\mathbb{Q}}
\newcommand{\Z}{\mathbb{Z}}
\newcommand{\N}{\mathbb{N}}
\newcommand{\A}{\mathbb{A}}
\newcommand{\pr}{\mathbb{P}}
\newcommand{\F}{\mathbb{F}}

\renewcommand{\O}{\mathcal{O}}
\renewcommand{\L}{\mathcal{L}}
\newcommand{\I}{\mathcal{I}}
\newcommand{\B}{\mathcal{B}}

\newcommand{\Mb}{\bold{M}}
\newcommand{\Nb}{\bold{N}}
\newcommand{\Eb}{\bold{E}}
\newcommand{\Fb}{\bold{F}}
\newcommand{\Tb}{\bold{T}}
\newcommand{\Db}{\bold{D}}
\newcommand{\Ub}{\bold{U}}
\newcommand{\Vb}{\bold{V}}
\newcommand{\Pb}{\bold{P}}

\newcommand{\Mbt}{\tilde{\bold{M}}}
\newcommand{\Nbt}{\tilde{\bold{N}}}

\newcommand{\Cbar}{\overline{C}}
\newcommand{\kCbar}{\overline{k[C]}}
\newcommand{\Ot}{\tilde{\mathcal{O}}}

\renewcommand{\P}{\mathcal{P}}
\newcommand{\p}{\mathfrak{p}}
\newcommand{\q}{\mathfrak{q}}

\newcommand{\m}{\mathfrak{m}}

\newtheorem{thm}{Theorem}[section]
\newtheorem{prop}[thm]{Proposition}
\newtheorem{definition}[thm]{Definition}
\newtheorem{coro}[thm]{Corollary}
\newtheorem{lem}[thm]{Lemma}
\newtheorem{rmq}[thm]{Remark}
\newtheorem{ex}[thm]{Example}

\DeclareMathOperator{\Supp}{Supp}
\DeclareMathOperator{\Div}{Div}

\DeclareMathOperator{\Tr}{Tr}
\DeclareMathOperator{\Spec}{Spec}
\DeclareMathOperator{\Proj}{Proj}

\DeclareMathOperator{\Card}{Card}

\DeclareMathOperator{\car}{char}
\DeclareMathOperator{\rdeg}{rdeg}

\DeclareMathOperator{\divi}{div}
\DeclareMathOperator{\ord}{ord}
\DeclareMathOperator{\Gal}{Gal}
\DeclareMathOperator{\Frac}{Frac}

\DeclareMathOperator{\diag}{diag}

\DeclareMathOperator{\red}{red}
\DeclareMathOperator{\Res}{Res}
\DeclareMathOperator{\disc}{disc}
\DeclareMathOperator{\com}{com}

\title{Fast computation of Riemann--Roch spaces for singular curves}
\author{Dounia Darkaoui}

\author{Martin Weimann}

\author{%
  Dounia Darkaoui \and Martin Weimann \\
}

\date{%
   LMNO, Universit\'e de Caen Normandie \\ \ \\
    \today
}

\begin{document}

\maketitle

\begin{abstract}
Let $\C$ be a projective curve defined over a field $k$ and let $D$ be a divisor of $\C$. The Riemann--Roch space $\L(D)$ is the set of rational functions on $\C$ for which certain zeros are imposed and certain poles are allowed, with some multiplicities determined by $D$. Riemann--Roch spaces play a fundamental role in algebraic geometry due to the central place of the Riemann--Roch theorem.  They have also important applications, such as coding theory or arithmetic of Jacobians of curves. 
In this article, we present what we believe is the fastest algorithm to date that computes a basis of a Riemann--Roch space for a curve with arbitrary singularities. Our algorithm is deterministic, works over any perfect field $k$, and works with no assumptions on the support of $D$.  
\end{abstract}



\section{Introduction}

Let $\C\subset \pr^2_k$ be an irreducible projective plane curve over a perfect field $k$, with function field  $L=k(\C)$. The Riemann--Roch  space attached to a divisor $D$ on $\C$ is $$\mathcal{L}(D):=\{b\in L \mid \divi(b)\geq D\}\cup \{0\}$$ where $\divi(b)$ stands for the principal divisor of $b$. The elements of $\L(D)$ are rational functions of $\C$ that are allowed to have some certain poles with prescribed multiplicities and are required to have certain zeros with prescribed multiplicities and $\L(D)$ is a $k$-vector space of finite dimension. 

The problem of computing Riemann--Roch spaces has attracted considerable interest since the 1980s,
when Goppa \cite{Goppa1,Goppa2,Goppa3} designed a new family of error-correcting codes (AG codes) obtained by evaluating functions of a Riemann--Roch space $\L(D)$ attached to a curve over a finite field (see \cite{Niederreiter} for a state of the art). The computation of particular Riemann--Roch spaces is also a cornerstone to perform operations in the Jacobian of the curve $\C$ \cite{Huang-Ierardi, Volchek}, with various applications in number theory and algebraic geometry. Among other applications of Riemann--Roch spaces, we can mention diophantine equations \cite{Coates}, integration of algebraic functions \cite{Davenport}, parametrization of rational curves \cite{Hoeij97}, or factorization of bivariate polynomials \cite{Wei}. 
Due to its omnipresence in algebraic geometry and its various applications, the problem of computing Riemann--Roch spaces became a fundamental task of computer algebra in the last four decades, as illustrated by an extensive literature, see e.g. \cite{LeBrigand-Risler,Hache,Volchek,PJ,ACL1,ACL2,ABCL, hess, schmidt, bauch, Ba16, Campillo}.

For curves with nodal or ordinary singularities, there exists now probabilistic algorithms of Las Vegas type with a complexity which is closed to be quasi-optimal \cite{ACL1,ACL2}. In this article, we present what we believe is the fastest algorithm to date that computes a basis of a Riemann--Roch space for a curve with arbitrary singularities (Theorem \ref{thm:main}). Our algorithm is deterministic, works over any perfect field $k$, with no assumptions on the support of $D$.  

We essentially follow Hess's algorithm \cite{hess}, based on the computation of integral bases of fractional ideals of Dedekind rings and on reduction of $k[t]$-modules. We revisit this method, by taking into account the recent developments in algorithmic of local and global fields, a cornerstone being the OM algorithm which allows to compute, represent and manipulate efficiently prime ideals of Dedekind domains \cite{Mo99,GuMoNa12,GuMoNa13,PoWe22}. 

Besides complexity issues, this article has a pedagogical focus. In particular, it intends to clear up the relations between the arithmetic point of view that we adopt here (divisors seen as a $\Z$-combination of prime ideals of Dedekind rings) and the geometric point of view that is often used in the literature (divisors seen as a $\Z$-combination of  points on a smooth model of a curve). Although this gymnastics is classical for experts, we hope this effort will help the non-expert reader understand the various methods of computing Riemann--Roch spaces.

\paragraph{Main result.} We use an algebraic RAM model as in \cite{Ka88}, counting only the number of arithmetic operations and zero tests in $k$. We denote $\O(g(n))$ and $\Tilde{\O}(g(n))$ to respectively hide constants and logarithmic factors with respect to $n$. We assume that fast multiplication of polynomials is used.  We denote $2\le \omega\le 3 $ the constant of linear algebra,  so that two $n\times n$ matrices over a commutative ring can be multiplied with $\O(n^\omega)$ ring operations. The best current bound is $\omega< 2.37286$ in \cite{omega}. We prove :

\begin{thm}\label{thm:main}
There exists a deterministic algorithm that, given  a projective curve $\C\subset \mathbb{P}^2_k$ of degree $n$ defined by an irreducible homogeneous polynomial $F\in k[X_0,X_1,X_2]$ monic and separable in $X_2$, and given a divisor $D\in \Div(\C)$ of non negative degree with effective part $D^+$, computes a compressed basis of $\mathcal{L}(D)$ with $$\Tilde{\O}(n^\omega(\deg(D^+)+\delta(\C)+n ))\,\, \subset \,\, \Ot(n^\omega(\deg(D^+)+n^2))$$
 operations in $k$, where $\delta(\C)=\frac{(n-1)(n-2)}{2}-g$, with $g$ the geometric genus of $\C$. 
\end{thm}

The divisor is represented as a $\Z$-combination of places of the function field $k(\C)$,  each place being given by an OM representation of a prime ideal of a suitable Dedekind subring (Definition \ref{def:OM rep}). A compressed basis (Definition \ref{def:compressed_basis}) is a list $(b_1,d_1),\ldots,(b_n,d_n)$ with $b_i\in k(t)[x]$ and $d_i\in \Z$ such that the $k$-vector space $\L(D)$ has basis
$$
\left\{ b_i(t,x) t^j, \,\, i=1,\ldots,n,\,\, j=0,\ldots,d_i \right\},
$$
where $t=X_1/X_0$ and $x=X_2/X_0$ are regarded modulo $F$.  If $D_\infty=\divi_\infty(t)$ is the divisor at infinity, we deduce for free the compressed basis of $\L(D+rD_\infty)$ for any integer $r\in \Z$ by simply replacing the integers $d_i$ with $d_i+r$. The cost of writting down all elements of a basis of $\L(D)$ is estimated in Proposition \ref{prop:expanded_basis}. 

Up to our knowledge, Theorem \ref{thm:main} improves significantly the best current bound $\Tilde{\O}((n^2 +\deg(D^+))^\omega)$ for curves with arbitrary singularities \cite{ABCL}. Moreover, we stress that our algorithm is deterministic, works regardless of the characteristic of the base field, and does not require $D$ to be supported at affine smooth points of $\C$, in contrast to  \cite{ABCL}. If the curve has ordinary singularities, our bound is less good by a factor $n$ than the (probabilistic) bound obtained in \cite{ACL1, ACL2}. Thus there is still room for improvements, a key point being that we represent elements of $\L(D)$ as rational functions with univariate denominators in contrast to \cite{ACL1, ACL2}.

The quantity $\delta(\C)$ is the so-called delta-invariant. It measures how far is $\C$ from being smooth: the less singularities the curve $\C$ has, the smaller $\delta(\C)$ gets and the faster our algorithm works. In particular, $\delta(\C)=0$ if (and only if) $\C$ is nonsingular. If the singular locus and the support of $D$ do not contain points at infinity, we get a slightly better complexity $\Tilde{\O}(n^\omega(\deg(D^+)+\delta(\C))$ (Theorem \ref{thm:main_recall}). 

A more precise statement of Theorem \ref{thm:main} is given in Theorem \ref{thm:RR}, with a complexity expressed in terms of finner (and possibly much smaller) invariants than  $\deg(D^+)$, which reflect how far is $D$ from being of shape $\divi(q)+r D_\infty$ with $q\in k(t)$ and $r\in \Z$ (Proposition \ref{prop:expI_vs_degD}).

\begin{rmq}
The hypothesis $F$ monic and separable with respect to $X_2$ can be freely replaced by the existence of two $k$-rational points $P,Q\in \pr_k^2$ that do not lie on $\C$ (Proposition \ref{prop:monic_separable}). If $\Card(k)>n$, this is not an issue. If $\Card(k)\le n$, it might happen that $\C$ passes through all $k$-rational points of $\pr_k^2$ (filling curves), in which case we have to work over a finite extension $k'$ of $k$ to reduce to the separable monic case. We compute a $k'$-basis of $\L(D)\otimes k'$ with the same complexity up to a logarithmic factor $[k':k]=O(\log(n))$. However, recovering a $k$-basis of $\L(D)$ from a $k'$-basis of $\L(D)\otimes k'$ involves probabilistic routines (see \cite{ACL2} and Section \ref{sec:monic_separable} for details).  
\end{rmq}

\paragraph{Main lines of the proof.} As said, we mainly follow Hess's algorithm \cite{hess} (see also Bauch's thesis \cite{bauch}). The function field $k(\C)$ is naturally isomorphic to a finite extension $L$ of a rational function field $k(t)$ defined by a separable, irreducible, monic polynomial $f\in k[t][X]$ of degree $n$.  
We can associate to $D$  two fractional ideals $I=I(D)$ and $I_\infty=I_\infty(D)$ of the respective integral closures of $k[t]$ and $k[t]_{(1/t)}$ in $L$. The fractional ideal $I$  is a free $k[t]$-module of rank $n$ which represents the affine part of the Riemann--Roch space $\L(D)$, while $I_\infty$ is a free $k[1/t]_{(1/t)}$-module of rank $n$ which represents the part at infinity.
We have $\L(D)=I\cap I_\infty$ and we are reduced to perform the following tasks : 

\noindent

$\bullet$ Step 1. Compute a $k[t]$-basis $\B$  of $I$ and a $k[1/t]_{(1/t)}$-basis $\B_\infty$ of $ I_\infty$.

\noindent

$\bullet$ Step 2. Compute a basis of the $k$-vector space $I\cap I_\infty$.

\noindent

For the first step, we use the recent fast algorithm in \cite{PoWe24} by Poteaux and the second author, following Okutsu's framework \cite{Ok82}. This algorithm is based on the combination of the fast OM algorithm in \cite{PoWe22} by the same authors (factorization of polynomials over local fields),  together with the remarkable MaxMin algorithm of Stainsby \cite{stainsby} which computes from an OM factorization a triangular $p$-basis of a given fractional ideal (see Section \ref{section:triang basis}).

The two  sets $\B$ and $\B_\infty$ form a basis of the $k(t)$-vector space $L$, and it is shown in \cite{hess} (see also \cite{bauch}) that step 2 reduces to compute a row reduced form of  the transition matrix $M\in k(t)^{n\times n }$ between $\B$ and $\B_\infty$. To this aim, we use fast algorithms for multiplying, inverting and reducing polynomial matrices of \cite{neiger-vu,stor13}, taking care of the particular shapes of the involved bases.

%



\paragraph{State of the art.}
Classical methods for computing bases of Riemann--Roch spaces are derived from the pioneer work of Brill and Noether which dates back to the second half of the 19th century. The Brill-Noether algorithm first computes a common denominator $H$ for all elements of $\L(D)$, thought as quotients of homogeneous polynomials of same degrees regarded modulo $F$. Writing $D=D^+-D^-$ as a difference of effective divisors with disjoint supports, Brill and Noether showed that it's sufficient to consider $\divi_0(H)\ge D^+ +\mathcal{A}$ where $\mathcal{A}$ is the so-called adjoint divisor. Given such an $H$, computing numerators amounts to compute a basis of the set of homogeneous polynomials $G$ (mod $F$) such that $\divi_0(G)\ge D-\divi_0(H)$. We have $D-\divi_0(H)\ge D^-+\mathcal{A}\ge 0$. Hence in both cases, we look for \emph{polynomials with prescribed zeros}. The Riemann--Roch theorem gives upper bounds for the degrees of the polynomials $G$ we are looking for, and after translating each condition $v_P(G)\ge n_P$ into a system of linear equations, the algorithm is mainly reduced to linear algebra. We usually talk about geometric methods, although \textit{linear algebra methods} also seems to be an appropriate terminology.

A first drawback of the Brill-Noether method is that it was originally developed for curves with nodal singularities, in which case the adjoint divisor is easy to describe. Modern algorithms with good complexities have been developed for nodal curves in \cite{PJ} using fast linear algebra, and improved in \cite{ACL1} using structured linear algebra based on $k[t]$-modules, leading to the best currently known complexity $\Tilde{\O}((n^2 +\deg(D^+))^{\frac{\omega +1}{2}})$ for nodal curves. This result was then generalized in \cite{ACL2} to curves with ordinary singularities. These algorithms are probabilistic of Las-Vegas type since they required various generic change of coordinates in order to compute a univariate representation of the involved divisors.  
Considering nodal curves is theoretically sufficient since any curve is birationally equivalent to a plane nodal curve (possibly after extension of the base field).
However, computing a nodal model of a singular curve is difficult and expensive. Moreover, considering curves with arbitrary (non ordinary) singularities led to the discovery of new AG-codes with excellent properties \cite{LeBrigand-Risler,TVZ,Niederreiter}. The Brill-Noether algorithm was thus extended to  curves with arbitrary singularities by various authors \cite{LeBrigand-Risler, Hache, Volchek, Campillo}, using blow-ups or local parametrizations of places. These methods have been recently improved in \cite{ABCL} under the assumption that $k$ has characteristic zero, in which case we can use the fast algorithm \cite{PoWe21} to represent singular places by Puiseux expansions. The algorithm of \cite{ABCL} is probabilistic, with complexity $\Tilde{\O}((n^2 +\deg(D^+))^\omega)$. Up to our knowledge, this was up to now the best complexity for curves with arbitrary singularities (assuming characteristic zero). 

A second drawback of the Brill-Noether method is that it supposes usually $D$ to be supported at smooth points. The fast algorithms \cite{PJ,ACL1,ACL2,ABCL} assume this. This hypothesis can not be dropped by a change of variable. Although this constraint is (up to our knowledge) not an issue for applications to AG codes, computing Riemann--Roch spaces attached to singular divisors is a relevant problem since various interesting divisors of a curve are precisely supported at singular places (different or co-different divisor, adjoint divisors, canonical divisor, etc.). For example, considering the canonical divisor has applications to integration of algebraic functions \cite{Davenport}, to parametrization of rational curves \cite{Hoeij97}, to factorization of bivariate polynomials \cite{Wei}, or to the computation of smooth canonical models of curves, from which one can extract various birational invariants such as the Clifford index or the gonality \cite{SSW}. 

Some alternative approaches that overcome the drawbacks of Brill-Noether algorithms have been proposed, with a more arithmetic flavour. Arithmetic methods do not require generic position (except maybe monicity and separability of the input polynomial), do not require $D$ to be supported at regular points, and do not make assumptions on the characteristic of the base field. They use integral bases of fractional ideals in function fields, in the spirit of the works of Coates \cite{Coates} and Davenport \cite{Davenport}. There has been several contributions in the last decades  \cite{schmidt, Matsumoto, bauch, hess}, Hess's algorithm \cite{hess} being a key reference since the years 2000s, implemented in the computer algebra systems \textsf{Magma} \cite{Magma} and \textsf{Singular} \cite{Singular1,Singular2}.  It is shown to run in polynomial time in \cite{hess}, although no complexity estimates are provided. In 2014, Bauch obtained complexity estimates in his thesis  \cite{bauch}, but the computation of integral bases in function fields was not yet competitive with the fastest linear algebra based algorithms of Brill-Noether type. Since then, there have been significant advances regarding computational aspects of function fields (and number fields). A key tool is the OM algorithm developed in 1999 by Montes \cite{Mo99}, following one century of work by Ore \cite{Ore23,Ore28}, MacLane \cite{Ma36a,Ma36b} and Okutsu \cite{Ok82}  and many others, and developed further by Nart and its collaborators \cite{BaNaSt13,GuMoNa10,GuMoNa11,GuMoNa12,GuNaPa12, GuMoNa13,GuMoNa15,AGNPRW} in the last two decades. Starting from a complete discrete rank one valued field $K$, a polynomial $f\in K[x]$  and the Gauss valuation on
$K[x]$, using Newton polygons and residual polynomials, the OM algorithm builds node after node a tree of valuations whose leaves correspond to the irreducible factors of $f$. It returns as a byproduct some extra data that allows to represent and manipulate effectively prime ideals in Dedekind domains. A version with quasi-optimal complexity has been recently developed in \cite{PoWe22}. Among various applications, the OM algorithm led to new algorithms to compute integral bases of fractional ideals \cite{Ba16, GuMoNa15, stainsby} and the quasi-optimal complexity estimates recently obtained in \cite{PoWe24} opened the door to make arithmetic methods competitive with geometric methods, as the present paper intends to show. 

Besides integral basis computations, an other feature of Hess's algorithm (and of fast Brill-Noether algorithms \cite{ACL1, ACL2, ABCL}) is to use basis reduction of free $k[t]$-modules, that is reduction of polynomial matrices. This problem of computer algebra has been considerably developed in the last two decades, and there are now fast algorithms for various tasks concerning polynomial matrices (multiplication, inversion, reduction, etc), see e.g. \cite{stor13,zhou,neiger-vu} and references therein. Some of these results are used in the present paper, and we believe there is still room to use more refined results to improve our complexity bounds.

\paragraph{Organization.}
In Section \ref{sec:2}, we gently introduce divisors and Riemann--Roch spaces in the language of function fields.
In Section \ref{sec:3}, we explain the relations between places and prime ideals of Dedekind rings and we define the OM representation of divisors. We mention too some other classical representations of divisors. Section \ref{sec:integral_basis} is dedicated to the computation of the integral bases of the fractional ideals $I$ and $I_\infty$ mentioned above. The complexity is naturally expressed in terms of invariants attached to $I$ and $I_\infty$, and then expressed in terms of $D$. In the key Section \ref{sec:5}, we compute the intersection $\L(D)=I\cap I_\infty$, a problem closely related to inversion and reduction of polynomial matrices. We prove Theorem \ref{thm:main} in Section \ref{sec:6}, relating the delta invariant of the projective curve $\C$ in terms of indices of integral closures. For the sake of completeness, we add two appendices : we first recall basic facts about integral closures (Annex \ref{sec:A}) and then explain why places of a function field $k(\C)$ correspond bijectively to closed points of the normalization of $\C$ (Annex \ref{sec:curves}).

\paragraph{Acknowledgments.} We thank Vincent Neiger for valuable discussions about reduction of polynomial matrices.

\section{Divisors and Riemann--Roch spaces}\label{sec:2}

A divisor on a smooth algebraic curve defined over an algebraically closed field  is usually defined as a finite formal $\Z$-combination of points of the curve. However, we want to consider singular curves defined over non algebraically closed fields, in which case the notion of point has to be replaced by the more subtle notion of place. 

\medskip

We fix $k$ a base field, assumed to be perfect. Typically $k$ is a finite field or a number field. Classical references for what follows are \cite{stichtenoth, Engler, samuel, serre}.

\begin{definition} 
A function field $L$ over $k$ (denoted $L/k$) is a finitely generated extension of $k$ of transcendence degree one.  A place of $L/k$ is the maximal ideal $P$ of a discrete valuation ring  $\O\subset L$ containing  $k$. We denote by $V_L$ the set of places of $L$. 
\end{definition}

Recall that a ring $\O\subset L$ is a valuation ring if for all $x\in L$, we have $x\in \O$ or $x^{-1}\in \O$. It is a local ring with maximal ideal $P=\O\setminus \O^\times$. The valuation ring $\O$ is discrete if moreover $P$ is principal. Since $L$ has transcendence degree one over $k$, all valuation rings of $L$ containing $k$ are discrete valuation rings \cite[Thm 1.1.6]{stichtenoth}. 

\begin{definition} Let $P$ be a place of $L$, with valuation ring $\O_P$. 
\begin{itemize}
\item The residue field of $P$ is $k_P=\O_P/P$. It is a finite extension of $k$. 

\item The degree of $P$ over $k$ is $\deg(P)=[k_P:k].$ 

\item  A uniformizer of $P$ is a generator of the principal ideal $P$.
\end{itemize}
\end{definition}

\begin{definition}
A valuation on a function field $L/k$ is a \emph{surjective} map $v:L\to \Z\cup\{\infty\}$ wich vanishes on $k$ and such that for all $b,b'\in L$, we have:
    \begin{enumerate}
         \item $v(b)=\infty \iff b=0$,
        \item $v(bb')=v(b)+v(b')$
        \item $v(b+b')\ge \min (v(b),v(b'))$ with equality if $v(b)\neq v(b')$.
    \end{enumerate}
\end{definition}

We can associate to each place a valuation. Let $P\in V_L$ with valuation ring $\O$ with uniformizer $\pi$. Any $b\in L^\times$ writes uniquely as $$b=u \pi^n,\quad u\in \O^\times,\quad n\in \Z.$$ The integer $n$ does not depend on the choice of the generator $\pi$ and the map $v_P(b)=n$ defines a valuation on $L$.  Conversely, a valuation $v$ on $L$  defines a valuation ring and its maximal ideal,
$$\O=\{b\in L \mid v(b)\ge 0\}\qquad {\rm and}\qquad P=\{b\in L \mid v(b)\ge 1\}$$
 and we  check that $v=v_P$. To summarize, there is a  one-to-one correspondence between the places of $L/k$, the (discrete) valuation rings of $L$ containing $k$ and the valuations on $L/k$ (see \cite[Thm 1.1.13]{stichtenoth} for further details). 

%
%


\begin{definition}
A divisor $D$ of $L/k$ is a finite formal $\Z$-combination of places of $L/k$, that is,  
$$
D= \sum_{P\in V_L} n_P\cdot P,  \qquad n_P\in \Z,\quad n_P\ne 0 \,\,\,{\rm for \,\,\, finitely \,\,\, many\,\,\, } P.
$$ 
We say that $D$ is effective if $n_P\ge 0$ for all $P$, denoted $D\ge 0$. \begin{itemize}
\item The degree of $D$ is $\deg (D)=\sum_P n_P\deg (P).$
\item The support of $D$ is $\Supp(D)=\{P\in L \mid n_P\ne 0\}.$
\end{itemize}
\end{definition}

\begin{rmq}
A divisor of a curve $\C$ defined over $k$ is simply a divisor of the field of rational functions $k(\C)$ of $\C$ (which is indeed a function field), see Section \ref{sec:6} for details. 
\end{rmq}

We omit the dependency in the base field $k$ and simply denote $\Div(L)$ the set of divisors of $L/k$. This is a group, namely the free abelian group generated by the set of places of $L/k$. This group admits a partial ordering by letting $E\ge D$ if $E-D\ge 0$. 

\begin{definition}
Let $b\in L^\times$. We have $v_P(b)\ne 0$ for finitely many places \cite[Cor. 1.3.4]{stichtenoth}. The divisor of $b$ is 
$$\divi(b)=\sum_{P\in V_L}v_P(b)\cdot P.$$ 
We say that $P$ is a zero of $b$ if $v_P(b)>0$ and that $P$ is a pole of $b$ if $v_P(b)<0$. A divisor of shape $\divi(b)$ is called a principal divisor.  
\end{definition} 

\noindent
For any $b\in \O_P$, we can "evaluate $b$ at $P$" by setting
$$
b(P):=b \,\,{\rm mod}\, P \,\,\in k_P,
$$ 
and $P$ is a zero of $b$ if and only if $b(P)=0$. If $b\notin \O_P$, we have $v_P(b)<0$ and we let $b(P)=\infty$. If $k$ is algebraically closed, then $k_P\simeq k$ for all $P$, and any $b\in L$ defines a  function $b:V_L\to k\cup \{\infty\}$. This explains the terminology "function field". 

\begin{definition} The Riemann--Roch space of a divisor $D$ is:
$$\mathcal{L}(D):=\{b\in L \mid \divi(b)+D\geq 0\}\cup\{0\}.$$
It is a $k$-vector space of finite dimension and $\mathcal{L}(D)=\{0\}$ if $\deg(D)< 0$ \cite[Prop. 1.4.9 and Cor. 1.4.12]{stichtenoth}.
\end{definition} 

We can write uniquely $D$ as a difference of effective divisors with disjoint support, namely $$D=D^+-D^-,\qquad {\rm where} \quad D^+=\sum_{n_P>0}n_P\cdot P,\quad {\rm and}\quad D^-=-\sum_{n_P<0}n_P\cdot P.$$ 
Likewise, for an element $b\in L$, we define $\divi_0(b):=\divi(b)^+$ \emph{the zero divisor} of $b$ and  $\divi_\infty(b):=\divi(b)^-$ \emph{the polar divisor} of $b$.
We thus have 
$$
b\in \mathcal{L}(D)\iff \,\, \divi_\infty(b)\leq D^+\quad {\rm and}\quad \divi_0(b)\geq D^-.
$$ 
In other words, we look for rational functions $b$ with at most that many poles and at least that many zeros.

\medskip
\noindent
We want to compute a $k$-basis of $\mathcal{L}(D)$.  The first questions are :

$\bullet$ How do we represent the field $L$ and its elements ?

$\bullet$ How do we represent a place $P\in V_L$ ? 

$\bullet$ How do we compute the valuation $v_P(b)$ of an element $b\in L$ ? 


\medskip
\noindent
There are various possible answers to these questions, which depend on how we think a place. There are usually two main points of view. The geometric point of view considers places as closed points of the desingularization of a projective plane curve $\C$. In this setting, places are usually represented using blow-ups \cite{LeBrigand-Risler}, Puiseux expansions \cite{ABCL, schmidt} or Hamburger-Noether expansions \cite{Campillo}. The arithmetic point of view, that we will adopt here,  considers places as prime ideals of Dedekind rings. In this setting, places can be represented by  generators of the prime ideal, by OM representation, or by expansions, as we will see in the next section. We refer the reader to Annex \ref{sec:curves} for the relations between these two points of view. 
%

\section{Representation of places of a function field}\label{sec:3}

We let $\Spec R$ be the set of prime ideals of a ring $R$.  For $p\in \Spec R$, we let $R_p$ be the localization of $R$ at $p$, that is by the multiplicative subset $R\setminus p$. 

\subsection{Places of a rational function field.}

\begin{definition} A function field $K$ is a rational function field over $k$ if $K=k(t)$ for some $t\in K$. 
\end{definition}

Let $K=k(t)$ be a rational function field. Let $A=k[t]$ and $p\in \Spec\,A$. Since $A$ is integrally closed of dimension one, the local ring $A_p$ is principal and integrally closed, hence is a discrete valuation ring of $K$, defining a place $P=pA_p$ of $K$. We say that $P$ is a finite place (with respect to the choice of the generator $t$). Note that $p$ is generated by a unique monic irreducible polynomial of $k[t]$ (abusively still denoted $p$) and the valuation $v_P$ coincides with the usual $p$-adic valuation $v_p$. 

There is an other valuation ring in $K$ containing $k$, the local ring $A_\infty=k[1/t]_{(1/t)}$, whose maximal ideal $(1/t)$ is called the place at infinity, denoted $P=\infty$.  The corresponding valuation $v_\infty$ is the minus degree valuation (take care of sign)$$v_\infty(g/h)=-\deg(g/h):=\deg(h)-\deg(g).$$ 

\begin{prop}\label{prop:palces_of_K}\cite[Thm 1.2.2]{stichtenoth}
The places of $K/k$ are exactly the places $P=pA_p$ with $p\in \Spec A$ and the place at infinity $P=\infty$. 
\end{prop}

There is thus a one-to-one correspondence between places of $K$ and $\P_A:=\Spec A\cup \{\infty\}$. Moreover, given $p\in \P_A$, the corresponding valuation $v_p$ is easy to describe. Note that geometrically, we may think $\P_A$ to be the set of closed points of the projective line $\pr^1_k=\A^1_k\cup\{\infty\}$.

%
%
%

\subsection{Places vs prime ideals of Dedekind rings.}

Let $L/k$ be an arbitrary function field. Let $t\in L$ be transcendental over $k$. Thus $L$ is a finite extension of the rational function field $K=k(t)$. 

\begin{lem}\label{lem:place_extension}
Let $P$ be a place of $L$. The restriction $P'=P\cap K$ is a place of $K$. We say that $P$ lies above $P'$.
\end{lem}

\begin{proof}
Clearly $\O_P \cap K$ is a valuation ring of $K$ which contains $k$. The maximal ideal of $\O_P$ is $P=\{x\in L,\, x^{-1}\notin \O_P\}$ and the  maximal ideal $P'$ is the set $P'=\{x\in K,\, x^{-1}\notin \O_P\cap K\}$, that is $P'=P\cap K$. 
\end{proof}

\begin{definition}\label{def:finite_places}
We say that a place of $L$ is finite if it lies above a finite place of $K$, and is infinite if it lies over $\infty$ (this notion depends on the choice of $t$). 
\end{definition}

Let $B$ and $B_\infty$ be the respective integral closures of $A$ and $A_\infty$ in $L$ (Definition \ref{def:integral_closure}). 
We denote for short 
$$\P_0=\Spec\,B,\quad \P_\infty=\Spec\,B_\infty \ \ \ \text{ and} \ \  \ \P=\P_0\cup \P_{\infty}.
$$ 
We call $\P$ the set of primes of $L/K$ (note that maximal ideals and non zero prime ideals coincide in our context). We say that $\p\in \P$  divides $p\in \P_A$ (or lies above $p$) if $p=\p\cap K$, denoted $\p|p$. 

\medskip

\begin{prop}\label{prop:Place_to_primes} Places of $L/k$ are in one-to-one correspondence with the primes of $L/K$. More precisely :
\begin{enumerate}
\item The set of finite places of $L$ is $V_L(B)=\{\p B_\p \mid \p\in \P_0\}$. 
\item The set of infinite places of $L$ is $V_L(B_\infty)=\{\p B_{\infty,\p} \mid \p\in \P_{\infty}\}$.
\item The places $P\in V_L$ lying above $pA_p\in V_K$ are one-to-one with the primes $\p\in \P$ lying above $p\in \P_A$. 
\end{enumerate}
\end{prop}

\begin{proof} 
The rings $B$ and $B_\infty$ being integrally closed, their localizations at their maximal ideals are distinct valuation rings of $L$, hence define distinct places of $L$. Conversely, if $P\in V_L$ is a finite place, then $P\cap K$ is a finite place of $K$ by Lemma \ref{lem:place_extension}, that is $\O_P\cap K=A_p$ for some prime $p\in \Spec A$ by Proposition \ref{prop:palces_of_K}. Thus $\O_P$ contains $A$, hence contains $B$ since $\O_P$ is integrally closed. Let $\p=P\cap B$. Then $\p\cap A=p$ and thus $\O_P$ contains $B_\p$. Since a discrete valuation ring is a maximal subring of its field of fractions, we get $\O_P=B_\p$ and $P=\p B_\p$. The same reasoning applies for places at infinity.
\end{proof}

If $P\in V_L$, we let $\p\in \P$ be the corresponding prime, called the prime of $P$. Conversely, if $\p\in \P$, we let $P\in V_L$ be the corresponding place. We have
$$
P=\p B_\p\ \ \ \text{ and } \ \ \ \p=P\cap B
$$
for finite places and primes, and similarly for the places at infinity. Recall the following fact :

\begin{lem}\label{lem:qBp=Bp}
If $\p,\q$ are distinct maximal ideals of an integral ring $R$, then $\p\cap\q=\p\q$ and $\q R_\p=R_\p$.
\end{lem}

\begin{proof}
As $\p,\q$ are distinct maximal ideals of $R$, we necessarily have $\p+\q=R$, i.e. $\p$ and $\q$ are coprime. Thus intersection and product coincide. Moreover, there exists $y\in \q$ such that $y\notin \p$. Thus $y$ is invertible in $R_\p$ and it follows that $1\in \q R_\p$, that is $\q R_\p=R_\p$. 
\end{proof}

\medskip
\noindent 

The ring $B$ is integrally closed of dimension one, hence is a Dedekind ring \cite[thm 9.3]{Atiyah}. Thus all fractional ideals of $B$ are invertible \cite[thm 9.8]{Atiyah}, the set of fractional ideals of $B$ is the free abelian multiplicative group generated by the primes $\p\in \P_0$ \cite[Cor. 9.8]{Atiyah}, and any fractional ideal $I$ of $B$ admits a unique factorization \cite[Cor 9.4]{Atiyah} $$I=\prod_{\p\in \P_0} \p^{n_\p},\quad \ \ \text{with} \ \ n_\p\in \Z \ \ \text{and} \ \ n_\p=0\ \ \text{for  almost all}\ \ \p.$$
We define $v_\p(I)=n_\p$, with convention $v_\p(0)=\infty$. This induces the $\p$-adic valuation $$v_\p: L\to \Z\cup\{\infty\},\quad v_\p(b):=v_\p(bB).$$ 
In the same way, $B_\infty$ is a Dedekind ring, and we  define analogously the $\p$-adic valuation attached to any prime $\p\in\P_{\infty}$.

%
%

\begin{lem}\label{lem:vP=vp}
Let $P\in V_L$ with prime $\p\in \P$. The valuations $v_P$ and $v_\p$ coincide.
\end{lem}

\begin{proof} 
By Lemma \ref{lem:qBp=Bp}, $v_\p(b)$ coincides with the largest integer $n$ such that $b\in \p^n$, and with the largest integer such that $b\in \p^{n}B_\p$. As $\p B_\p=P$, the claim follows.
\end{proof}



Let us conclude this section with few classical definitions and results.

\begin{definition}\label{def:residual_degree}
Let $\p\in \P$ be a finite place and let $p=\p\cap K$. 
\begin{itemize}
\item The degree of $\p$ is $\deg(\p):=[B/\p B :k]$.
\item The residual degree of $\p/p$ is  $f_\p:=[B/\p B: A/pA]$.
\item The index of ramification of $\p/p$ is $e_\p:=v_\p(p)$. 
\end{itemize}

\end{definition}


We extend trivially this definition for places at infinity, replacing $A,B$ by $A_\infty,B_\infty$. 
Note that in such a case, $p=\infty$ and $A_\infty/pA_\infty=k$. Thus, degree and residual degree coincide.  If $p$ is finite, we identify $p$ with its unique monic generator $p\in A=k[t]$. 

\begin{prop}\label{prop:degree_place}
Let $P$ be a finite place of $L$ with prime $\p$ lying above $p\in \Spec A$. We have  $$\deg(P)=\deg(\p)=f_\p \deg(p).$$
\end{prop}

\begin{proof}
First equality is immediate since $\O_P/P = B_\p/\p B_\p\simeq B/\p B$.
Second equality follows from the multiplicative property $[B/\p B: k]=[B/\p B: A/pA][A/pA:k]$ together with $[A/pA:k]=\deg(p)$.
\end{proof}

The following result is known as the fundamental inequality \cite[Thm 3.3.4 and Thm. 3.3.5]{Engler} :

\begin{prop}\label{prop:sum_ep_fp}
We have $\sum_{\p|p} e_\p f_\p \le [L:K]$, and equality holds if the extension is separable. 
\end{prop}

\subsection{OM-representations of places}

We assume from now on that $L=K[X]/(f)$ for some irreducible monic separable polynomial $f\in A[X]$.   
We let $n=\deg(f)=[L:K]$. Denoting $x=X\mod f$, we can write
$$
L=K(x)=k(t,x)=k(t)[x]
$$ 
and any $b\in L$ writes uniquely as $b=a_0 + a_1 x +\cdots + a_{n-1} x^{n-1}$, with $a_i\in k(t)$.

\medskip
By Proposition \ref{prop:Place_to_primes}, it's enough to represent the primes of $L/K$ to represent the places of $L/k$. Given $p\in \P_A$, we denote $\widehat{A}_p$ the $p$-adic completion of $A_p$. 

\begin{prop}\label{prop:primes_vs_local_factors}
There is a one-to-one correspondence between the primes $\p$ dividing $p$ and the monic irreducible factors $F_\p$ of $f$ in $\widehat{A}_p[X]$. 
\end{prop}

\begin{proof}
This follows from \cite[Prop. 8.2]{Neukirch} combined with Proposition \ref{prop:Place_to_primes}. 
\end{proof}


Given any $p\in \P_A$, the OM-algorithm allows to compute the factors $F_\p$ up to an arbitrary $p$-adic precision. This powerful algorithm has been developed by Montes \cite{Mo99} in the continuation of the pioneer work of Ore \cite{Ore23, Ore28}, MacLane \cite{Ma36a, Ma36b} and Okutsu \cite{Ok82}. In the last decade, it has been improved and generalized in various papers, see e.g. \cite{GuMoNa11, GuMoNa12, GuMoNa13, GuMoNa15, BaNaSt13, GuNaPa12, PoWe22, AGNPRW}. 

\medskip

Briefly speaking, the OM-algorithm detects at each iteration a partial factorization of $f$ on a higher order Newton polygon and on a higher order residual polynomial  (double dissection process), until the factorization is complete. It returns as a byproduct \emph{a type} $t_\p$ (also called Okutsu frame, or OM-representation)  attached to each prime $\p$. This object is a sequence 
\begin{equation}\label{eq:type}
t_\p=[p,\phi_{1},\ldots,\phi_{r_\p},\phi_{\p}]
\end{equation}
where $p=\p\cap A$ and where the $\phi_{i}\in A[X]$ are some particular irreducible monic polynomials  of strictly increasing degrees $d_1|\cdots|d_{r_\p}$, called key polynomials. The last key polynomial $\phi_\p$ is a $p$-adic approximation of $F_\p$ with a high enough precision which allows to distinguish $F_\p$ from the remaining factors $F_\q$ of $f$, see e.g. \cite[Sec.3]{BaNaSt13} for details.

\begin{rmq}
Key polynomials have been introduced by Mac Lane in the 1930s \cite{Ma36a,Ma36b}. When $k$ has characteristic zero or large enough, some appropriate "approximate roots" of $f$ may play the role of key polynomials. Approximate roots were originally introduced by Abhyankhar \cite{Ab89} to study germs of complex plane curves without using Puiseux series  (see e.g. \cite{Pop02} for a nice survey on this topic).
\end{rmq}

Let us give two simple examples. 

\begin{ex}\label{ex:OM_representation} 
$\bullet$ Let $f=X^3-X^2+t^2$ and let $p=(t)$. We assume that $\car(k)\ne 2$. The factorization of $f$ in $k[[t]][X]$ is $f=f_1 f_2 f_3$ with $f_1=X-t+O(t^2)$ , $f_2=X+t+O(t^2)$ and $f_3=X-1+O(t^2)$. Some OM-representations of the primes $\p_1,\p_2,\p_3$ lying above $(t)$ are given by
$$
t_{\p_1}=[t,X-t],\qquad t_{\p_2}=[t,X+t],\qquad t_{\p_3}=[t,X-1].
$$
\noindent
$\bullet$ Let $f=(X^2+1)^2+tX+t+t^2$. This polynomial is irreducible in $k((t))[X]$. Hence, there is a unique prime $\p$ lying above $(t)$. An OM-representation is
$$
t_{\p}=[t,X,X^2+1,(X^2+1)^2+t].
$$
\end{ex}

In practice, the types $t_\p$, $\p|p$ come together with extra numerical data which allows to calculate various arithmetic quantities (e.g. ramification, residual degree, index, intersection multiplicities), see e.g. \cite{GuMoNa13}. An important fact for us is that these data allow for a fast computation of the valuation $v_\p(b)$ of any $b\in L$. Concretely, this task is reduced to compute the $(\phi_{1},\ldots,\phi_{r_\p})$-multiadic expansion of $b$, which can be achieved in quasi-linear time. An other important feature of the OM-representation is that it allows also for a quick computation of an integral basis of a fractional ideal (see \cite{GuMoNa15, PoWe24} and Section \ref{sec:integral_basis}), which is a key point to compute Riemann--Roch spaces using arithmetic approach. 

\medskip

Let us insist that an OM-representation of $\p$ is given \textit{a priori} at small precision. If needed, we can compute the $\phi_\p$'s up to an arbitrary  precision using a single-factor lifting procedure \cite{GuNaPa12} or a multi-factor lifting procedure \cite{PoWe22}, the latter approach having a quasi-linear complexity thanks to a valuated Hensel's lemma. Higher precisions are required for various tasks such as computing valuations, computing two-elements representation of ideals, or computing integral bases of fractional ideals. 

\medskip

The OM-algorithm has been improved in \cite{PoWe22}. This fast version computes an OM-representation of all primes $\p$ dividing $p$ with $\Ot(n \delta_p)$ operations in the residue field $k_p$ if the residual characteristic is zero or high enough - this is quasi-linear in the size of the output -, or $\Ot(n \delta_p + \delta_p^2)$ operations in $k_p$ otherwise, where $\delta_p\le  v_p(\disc(f))$ is the so-called Okutsu invariant. We refer the reader to \cite{BaNaSt13, PoWe22, PoWe24} for more details about complexity issues.

\subsection{OM-representation of divisors}

\begin{definition}\label{def:OM rep}
Let $D=\sum_P n_P \cdot P$ be a divisor of $L$. An OM-representation of $D$ is a list 
$$
t_D=\{(t_P,n_P),\,P\in \Supp(D)\}
$$
where $t_P=t_\p$ is an OM-representation (a type as in \eqref{eq:type}) of the prime $\p\in \P$ associated to $P$. 
\end{definition}

Let us insist on the fact that the computation of the Riemann--Roch space $\L(D)$  does not involve only the places $P\in \Supp(D)$~: we need also to control that $v_P(b)\ge 0$ at all remaining places $P\notin \Supp(D)$. In particular, this requires to compute an OM-representation of all places centered at singular points of the projective curve defined by $f$. We do not assume that this data is given. However, this task has to be done only once if we need to compute Riemann--Roch spaces attached to several divisors.

\begin{rmq}
In many interesting situations, we may want  to compute the OM-representation of a divisor $D$ of $L$ which is of particular interest, and simply defined by its intrinsic properties. For instance the canonical divisor of a curve, the adjoint divisor of a curve, the conductor of $L/K$, the different, the codifferent, etc. For such divisors, there are explicit formula for the values $v_\p(D)$ in terms of the OM-invariants of $f$  (see e.g. \cite{Na14}), and it follows from \cite{PoWe22} that we can compute an OM-representation of $D$ in quasi-linear time.
\end{rmq}

\subsection{Other usual representations of places}

For the sake of completeness, let us mention two other representations of places which are classically used in the literature.

\subsubsection{Representation of places by Puiseux expansions.} \label{ssec:parametrization}

This representation of places is widely used in algebraic geometry, in particular for computing Riemann--Roch spaces \cite{ABCL}. Geometrically, the irreducible factor $F_\p$ of $f$ gives the local equation of a branch  $C_\p$ of the affine plane curve $C$ defined by $f$. In this setting, the valuation $v_\p$ corresponds to the intersection multiplicity with $F_\p$, which is usually computed by means of resultants. Namely :

\begin{prop}\label{prop:valuation_via_resultant} Let $h\in L=K[x]$. The $\p$-adic valuation of $h$ is given by 
\begin{equation}\label{eq:resultant}
v_\p(h)=\frac{v_p (\Res(h,F_{\p}))}{f_\p},
\end{equation}
where we still denote by $h$ the canonical lifting of $h$ to $K[X]$. 
\end{prop}

\begin{proof}
Let $\theta$ be a root of $F_\p$ in a fixed algebraic closure $\K$ of the $p$-adic completion of $K$. By the henselian property, the valuation $v_p$ extends uniquely to  a valuation $\bar{v}$ on $\K$, and we have the equality 
$v_\p(h)=e_\p \bar{v}(h(\theta))$ 
(see e.g. \cite[p.165]{Neukirch}). Let $d_\p=\deg F_\p$ and let $\theta_1=\theta,\theta_2,\ldots,\theta_{d_\p}$ be the roots of $F_\p$. By the Poisson formula for the resultant, we get
$v_p(\Res(h,F_{\p}))=\sum_{i=1}^{d_\p} \bar{v}(h(\theta_i))=d_\p \bar{v}(h(\theta))
$
and we conclude thanks to the equality $d_\p = e_\p f_\p$. 
\end{proof}

However, using resultant is not convenient in practice. An alternative way is to represent the prime $\p$ by a \emph{rational parametrization} of the corresponding branch $C_\p$, 
 that is by a pair $(a,b)\in k_\p((u))$ (which is truncated in practice) such that $F_\p(a(u),b(u))= 0$. If the characteristic of $k$ is zero or big enough, such a parametrization can be given by a "rational Puiseux expansion" of shape $(\lambda u^{e},S(u))$. The equality \eqref{eq:resultant} becomes 
\begin{equation}\label{eq:puiseux}
v_{\p}(h)=\ord_u h(\lambda u^{e},S(u)),
\end{equation}
which can be calculated quickly by Horner's rule \cite{ABCL}. Representing places as Puiseux expansions is particularly convenient for computing Riemann--Roch spaces since they allow to translate the conditions $v_\p(h)\ge n_\p$ into a finite system of linear equations over $k$ once we have a bound on the total degree of $h$ (see e.g. \cite{ABCL}). However, the complexity of the linear system solving is higher than our complexity bound for general singular curves. Nevertheless, using some extra $k[t]$-module structures (as we will do in this paper) it is shown in \cite{PJ, ACL1,ACL2} that we can fasten the computations for nodal or ordinary curves in general position.

\begin{rmq}
If the characteristic of $k$ is small, we can replace Puiseux  expansions by Hamburger-Noether parametrizations \cite{Ca80, Campillo}. However, we are not aware of any implementations and complexity estimates for this task. 
\end{rmq}

\begin{rmq}
The OM-algorithm computes the truncated factors $F_\p$, but does not allow to compute a parametrization of the corresponding branch. We believe that this is a nice open problem.
\end{rmq}

\subsubsection{Two-elements representation of a place.} It is based on the following classical result :

\begin{lem}\label{lem:two-elements-place}
Any  prime ideal $\p\subset B$ can be generated by two elements 
$\p=(p,q)$ where $p\in A$ is the monic generator of $\p\cap A$ and $q\in B$. If $\p\subset B_\infty$, there exists $q\in B_\infty$ such that $\p=(1/t,q)$. 
\end{lem}

This is probably the most classical way to represent prime ideals in function fields or number fields, as used for instance in the computer algebra systems \textsf{SageMath} or \textsf{PariGP}.

\begin{definition} The pair $(p,q)$ is called a two-elements representation of the prime $\p$ (or of the place $P$). 
\end{definition}

The $\q$-adic valuation of a fractional ideal equals the minimal valuation of its generator. In particular, if $\p=(p,q)$, we have $\min(v_\p(p),v_\p(q))=v_\p(\p)=1$ so $p$ or $q$ has to be a uniformizer of the place $P=P_\p$. However, this condition is not sufficient.  Indeed, the generator $q\in L$ needs to ensure that all conditions
$$
v_\p(\p)= 1 \quad \text{and}\quad v_\q(\p)= 0\, \,\,\, \forall \,\, \q \in \P_0,\,\, \q\ne \p
$$
are satisfied. This imposes conditions on $q$ which involve  \emph{a priori} all finite places of $L$.

\begin{ex}\label{ex:two_elements_representation} Let $f=X^3-X^2+t^2$ as in example \ref{ex:OM_representation}, with factors $f_1=X-t+O(t^2)$ , $f_2=X+t+O(t^2)$ and $f_3=X-1+O(t^2)$ above $p=(t)$. 
We claim that 
$$
\p_1=(t,q)\qquad \mathrm{where} \qquad q=\frac{(x-1+t)(x-t)}{t}.
$$
Let $\p=(t,q)$. Note that $v_\q(\p)=\min (v_\q(t),v_\q(q))$ for all finite primes $\q$. We need to check that $v_{\p_1}(\p)=1$ and $v_{\q}(\p)=0$ for all $\q\ne \p_1$.  If $\q$ does not divide $t$, then $v_\q(t)=0$ and $v_\q(q)\ge 0$ (the inequality since $x\in B\subset B_\q$ as $f$ is monic). Hence, $v_\q(\p)=0$. If $\q=\p_i$, we have $v_{\p_i}(t)=1$. Let us compute $v_{\p_i}(q)$. For $i=1$, we compute $f_1$ with a higher precision $f_1=X-t-t^2/2+O(t^3)$. Proposition \ref{prop:valuation_via_resultant} gives
$$
v_{\p_1}(q)=v_p\left(\Res((X-1+t)(X-t),X-t-\frac{t^2}{2}+O(t^3)\right)-1=2-1=1.
$$
Hence $v_{\p_1}(\p)=1$. For $i=2,3$ we find that 
$$
v_{\p_2}(q)=v_p(\Res((X-1+t)(X-t),X+t+O(t^2))-1=1-1=0
$$
and similarly $v_{\p_3}(q)=0$. Hence $v_{\p_2}(\p)=v_{\p_3}(\p)=0$ as required.  In a similar way, we would obtain the two-elements representations $\p_2=(t,(x-1+t)(x+t)/t)$ and $\p_3=(t,x-1)$. 
\end{ex}

More generally, we can deduce easily a two-elements representation of a prime $\p$ from its OM-representation, see \cite{GuMoNa13}. However, two-elements representations are not suitable for rapid calculation of the $\p$-adic valuation and do not contain enough information to compute integral bases. For this reason, we prefer OM-representations, although we do not exclude that mixing both representations could be of interest for various tasks, such as computing the integral basis of a product of two fractional ideals.

\section{From divisors to integral bases}\label{sec:integral_basis}

We keep notations of previous section. We still assume that $L=K[X]/(f)$ for some irreducible monic separable polynomial $f\in A[X]$ of degree $n$. 

\subsection{Divisors as pairs of fractional ideals}

Let $D=\sum_{P\in V_L} n_P\cdot P$ be a divisor of $L/k$. Denote $\p\in \P$ the prime of $P$ and let $n_\p=n_P$. We define 
\begin{equation} \label{eq:ideaux}
I(D)=\prod_{\p\in \P_0} \p^{-n_\p} \ \ \ \ \text{ and} \ \ \ \ I_{\infty}(D)=\prod_{\p\in \P_\infty} \p^{-n_\p}.
\end{equation}
Thus $I(D)\subset L$ is a fractional ideal of $B$ and $I_\infty(D)$ is a fractional ideal of $B_\infty $. Let $\I$ (resp. $\I_\infty$) stand for the multiplicative groups of fractional ideals of $B$ (resp. $B_\infty$), and let $\I\times\I_\infty$ be their direct product.  

\begin{prop}\label{prop:LD=IcapIinfty}
The map $D\mapsto (I(D), I_{\infty}(D))$ is a group isomorphism from $\Div(L)$ to $ \I\times\I_\infty$. Moreover, we have $$\L(D)=I(D)\cap I_{\infty}(D).$$
\end{prop}

\begin{proof}
First point is clear from Proposition \ref{prop:Place_to_primes} together with the fact that $\I$ and $\I_\infty$ are the free abelian groups generated by $\P_0$ (resp. $\P_\infty$). For the second point, we remark that  $\mathcal{L}(D)=\cap_{P\in V_L} P^{-n_P}$ by definition. 
By Proposition \ref{prop:integrally-closed=DVRintersection}, we have $B=\bigcap_{\p\in \P_0} B_{\p}$ and $B_{\infty}=\bigcap_{\p\in \P_\infty} B_{\infty,\p}$. It thus follows from Lemma \ref{prop:Place_to_primes} that $\L(D)$ is the intersection of the fractional ideal $\bigcap_{\p\in \P_0} \p^{-n_\p}$ of $B$ with the fractional ideal $\bigcap_{\p\in \P_\infty} \p^{-n_\p}$ of $B_\infty$. By Lemma \ref{lem:qBp=Bp} these intersections can be replaced by a product. 
\end{proof}

We want to compute a basis of $\L(D)$ using Proposition \ref{prop:LD=IcapIinfty}. The first step is to compute an integral basis of the fractional ideals $I(D)$ and $I_\infty(D)$. To this aim, we use recent results of \cite{PoWe22}. 

\subsection{Triangular bases of fractional ideals}

Let $R\subset K=k(t)$ be  a Dedekind ring with fraction field $K$ (typically $R=A$ or $R=A_\infty$).

\begin{definition}
    Let $M,N\subset L$ be free $R$-modules of rank $n$. We denote $[M:N]$ the fractional ideal of $R$ generated by the determinant of the transition matrix of an $R$-basis of $N$ to an $R$-basis of $M$. It is called the index ideal of $M$
 over $N$.
 \end{definition}
 
The definition does not depend on the choice of the bases since the determinant varies by a unit of $R$. The index ideal obeys to the following classical properties (see e.g. \cite{serre}):

\begin{lem}\label{lem:index}
    Let $L,M$ and $N$ be free $R$-modules of rank $n$. We have
    \begin{enumerate}
        \item $[L:N]=[L:M][M:N]$ and $[M:N]=[N:M]^{-1}$,
        \item If $N\subset M$, then there exists $a_1,\ldots ,a_n\in R$ such that  $a_1|\cdots |a_n$ and $$M/N\cong R/a_1 R\times \cdots \times R/a_n R.$$
        We have then $[M:N]=(a_1\cdots a_n)\subset R$ and $N=M$ if and only if $[M:N]=R$.
    \end{enumerate}
\end{lem}

\medskip
\noindent
Let $\overline{R}\subset L=k(t)[x]$ be the integral closure of $R$ in $L$ (typically $\overline{R}=B$ or $\overline{R}=B_\infty$). Thus $\overline{R}$ is again a Dedekind ring. As $L/K$ is assumed to be separable, all fractional ideals of $\overline{R}$ are free $R$-module of rank $n=[L:K]$. We refer to \cite[Thm.1.16]{stainsbyThesis} for the following result:

\begin{prop}\label{prop:triang_basis}
Any fractional ideal $I$ of $\overline{R}$ admits an $R$-basis 
of shape $\B=\left(q_0, q_1 h_1(x),\ldots,q_{n-1} h_{n-1}(x)\right)$
such that the following conditions are satisfied :
\begin{enumerate}
        \item For all $1\leq i <n$, $q_i\in K$ and $h_i(x)\in R[x]$ is monic of degree $i$. 
        \item We have inclusions $q_{0}R\subset q_{1}R\subset \cdots \subset q_{n-1}R$ of fractional ideals of $R$. 
    \end{enumerate}
We say that $\B$ is a triangular $R$-basis of $I$. The elements $q_i$ are uniquely determined up to multiplication by a unit of $R$ and we have
$$
[I:R[x]]=(q_0\cdots q_{n-1})^{-1} R.
$$ 
\end{prop}

\begin{rmq} Any free $R$-module $I\subset L$ of rank $n$ admits an $R$-basis satisfying the first condition by Gaussian elimination, but the second condition may not be satisfied. It turns out that if $I$ is a fractional ideal of a Dedekind ring, then condition 2 follows  automatically from condition 1.
\end{rmq}

\noindent
\textbf{Notation.} For $R=A$ or $R=A_\infty$, any fractional ideal $J$ of $R$ is principal, hence generated by a unique element $h=a/b$ with $a,b\in k[t]$ monic and coprime. We then write for short 
$$
\deg(J):=\deg(h)=\deg(a)-\deg(b).
$$

\subsection{Computing triangular bases over finite places} \label{section:triang basis}

We consider here $R=A=k[t]$. We want to compute an integral basis of a fractional ideal $I$ of $B$.  The index of $I$ does not reflect well the size of $I$, and we will rather express the complexity of our algorithm in terms of the index of a normalized ideal of $I$.

\begin{definition}\label{def:height}
Let $I$ be a fractional ideal of $B$. Let $q_I\in K$ such that $I\cap K=q_I A$.
\begin{itemize}
\item The normalized ideal of $I$  is $I^*=q_I^{-1}I$.
\item The normalized index of $I$ is $\delta_I=\deg\,[I^*:A[x]]$. 
\item The normalized exponent of $I$  is  $\exp(I):=\min\{\deg(a)\mid a\in A, aI^*\subset A[x]\}$.
\end{itemize}
\end{definition}

\noindent

\begin{lem}\label{lem:index exp} We have  $A[x]\subset B\subset I^*$. In particular, $ 0\le \delta_B\le \delta_I$ and $0\le \exp(B)\le \exp(I)$. Moreover,
$$
\exp(I)\le \delta_I\le n\exp(I).
$$
\end{lem}

\begin{proof} As $f$ is monic, $x$ is integral over $A$, hence $A[x]\subset B$. We have $I\cap K=q A$, hence $I^*\cap K=q^{-1} q A=A$. Thus $1\in I^*$ and $B\subset I^*$. 
From these inclusions, we deduce $0\le \exp(B)\le \exp(I)$, and the inequality $ 0\le \delta_B\le \delta_I$ follows from Lemma \ref{lem:index}, using $[I^*:A[x]]=[I^*:B][B:A[x]]$ and $A[x]\subset B\subset I^*$. Still from Lemma \ref{lem:index}, we get a decomposition $I^*/A[x]\simeq A/a_1 A \times\cdots \times A/a_n A$ with $a_i\in A$ and $a_i|a_{i+1}$. We deduce that $\exp(I)=\deg a_n$. Moreover, $[I^*:A[x]]=a_1\cdots a_n A$, hence  $\delta_I=\deg a_1+\cdots +\deg a_n$. Since $\deg a_i\le \deg a_{i+1}$, the last claim follows.
\end{proof}

\medskip
Computing integral bases in function fields is a classical task of computer algebra and there are several algorithms to do this. There are geometric approaches based on Puiseux series (see \cite{Abe} for a state-of-the art in the case $I=B$) and there are arithmetic approaches (\cite{Ba16, GuMoNa15, stainsby, PoWe24} and references therein) which rather follow Okutsu's framework \cite{Ok82} and use the OM-algorithm as we do here. This latter approach has the advantage to work in any characteristic and applies \emph{mutatis mutandis} to the case of number fields. Let us briefly summarize the main steps. Let $I=\prod_{\p\in \P_0} \p^{n_\p}$ be a fractional ideal of $B$.

\begin{enumerate}
\item For each finite prime $p\in \P_A$, we compute a triangular $A_p$-basis of the free $A_p$-module $I_p=I\otimes_A A_p$.  If $p$ does not divide the index $[B:A[x]]$ and $n_\p=0$ for all $\p|p$, then a triangular basis is $1,x,\ldots,x^{n-1}$. Otherwise:
\begin{enumerate}
\item We run the OM-algorithm \cite{PoWe22} above $p$ to compute an OM-representation of each $\p|p$ with high enough precision (i.e. a list of key polynomials of the prime factor $F_\p\in \widehat{A}_p[x]$ of $f$).
\item  We apply the simple and powerful MaxMin algorithm of Stainsby \cite{stainsby} to detect for each $i=0,\ldots,n-1$  a degree $i$ multiplicative combination $h_i$ of key polynomials which maximizes  a certain semi-valuation $w_{I,p}$.
Then $h_0/p^{\lfloor w_{I,p}(h_0) \rfloor},\ldots,h_{n-1}/p^{\lfloor w_{I,p}(h_{n-1}) \rfloor}$ is a triangular $A_p$-basis of $I_p$, see \cite[Thm.3.4]{stainsby}. 
\end{enumerate}
\item We use the Chinese Remainder Theorem to glue together all $A_p$-bases, resulting in a triangular $A$-basis of $I$ (we use here the fact that the $A_p$-bases are triangular).
\end{enumerate}


We use notations $\O_\epsilon(g(n))=\O(g(n)^{1+\epsilon(n)})$ with $\epsilon(n)\to 0$. We denote for short $\delta=\delta_B$. We obtain the following complexity estimate :
\begin{thm}\label{thm:complexity_integral_basis}\cite[Thm.7]{PoWe24}
    Suppose $f$ separable and monic. Let $I=\prod \p^{n_\p}$ be a fractional ideal of $B$, the primes $\p$ being given in OM-representation. There is a deterministic algorithm which computes a triangular $A$-basis of $I$ with
    \begin{itemize}
        \item $\O_\epsilon(n\delta_I)$ operations in $k$ if $char(k)=0$ or $char(k)> n$,
        \item $\O_\epsilon(n\delta_I +\delta^2)$ operations in $k$ if $char(k)\leq n $.
    \end{itemize}
\end{thm}

\medskip
\noindent
 The returned basis is given as a pair 
$\B=(q,\B^*)$, with $q=q_I\in k(t)$ as in Definition \ref{def:height} and $\B=q\B^*$, where
\begin{equation}\label{eq:triang_basis}
\B^*=\left(1, \frac{g_1(t,x)}{p_1(t)},\ldots , \frac{g_{n-1}(t,x)}{p_{n-1}(t)}\right)
\end{equation}
is an $A$-basis of the normalized ideal $I^*$ such that :

(i) $p_i\in k[t]$ is monic, $p_1|p_2|\cdots |p_{n-1}$. 

(ii) $ g_i\in k[t][x]$ is monic of degree $i$ in $x$ and $\deg_t(g_i)< \deg(p_i)$. 

\medskip
\noindent
The $p_i$'s are uniquely determined. We have
 $\delta_I=\deg(p_1)+\cdots +\deg(p_{n-1})$ and  $\exp(I)=\deg(p_{n-1})$.  Not taking into account the factor $q$, the output has arithmetic dense size $\sum_{i=1}^{n-1} i\deg(p_i)=\O(n\delta_{I})$, and this bound is sharp. Thus the algorithm has a quasi-linear complexity in characteristic zero or high enough.  The element $q=q_I$ is given as a product of primes $q=\prod_{p\in \P_A} p^{-m_p}$ (not computed), with the $m_p$'s explicitly determined by $I$, see below.

\begin{rmq}\label{rem:integral_basis}
Theorem 7 of \cite{PoWe24} requires to know the set of primes $p\in \P_A$ dividing $[I^*:A[x]]$. A part of these primes is deduced for free from the OM-representation of $I$ (which is given). The remaining primes are the prime divisors of $[B:A[x]]$. They can be deduced from the irreducible factorization in $k[t]$ of the discriminant of $f$. Using dynamic evaluation \cite{HoLe19}, it's in fact sufficient to compute the square-free factors of the discriminant, whose cost fits in the aimed bound.
\end{rmq}

\subsection{Computing triangular bases at infinity}\label{ssec:integral_basis_at_infinity}

We now specialize Proposition \ref{prop:triang_basis} to the case $R=A_\infty=k[1/t]_{(1/t)}$. We want to compute a triangular $A_\infty$-basis 
of a fractional ideal $I_\infty$ of $B_\infty$.  
We need to take care that $A_\infty[x]$ is usually not contained in $B_\infty$. Let $f=\sum_{i=0}^{n} c_iX^i$ with $c_i\in A=k[t] $. Following \cite{hess}, we consider
\begin{equation}\label{eq:lambda}
y=x t^{-\lambda}\qquad \mathrm{where}\qquad \lambda=\lambda(f):=\max_{i<n}\left(\left\lceil \frac{\deg(c_i)}{n-i}\right \rceil\right ).
\end{equation}
\begin{lem}\label{lem:y}
The element $y\in L$ is integral over $A_\infty$.
\end{lem}

\begin{proof}
The polynomial $f_\infty(Y):=t^{-n\lambda}f(t^\lambda Y)$
is a monic polynomial of degree $n$ which by \eqref{eq:lambda} belongs to $k[t^{-1}][Y]$. Since $f_\infty(y)=t^{-n\lambda} f(x)=0$ and $k[t^{-1}]\subset A_\infty$, we conclude. 
\end{proof}

Let us denote $u=t^{-1}$. Note that $\deg_u(a)=-\deg(a)$ for all $a\in K$. The ideal $I_\infty\cap K$ is now a fractional ideal of $A_\infty=k[u]_{(u)}$. 

\begin{definition}\label{def:height infinity}
Let $I_\infty$ be a fractional ideal of $B_\infty$ and let $m_{I_\infty}\in \Z$ such that $I_\infty\cap K=u^{m_{I_\infty}} A_\infty$. 
\begin{itemize}
\item The normalized ideal of $I_\infty$  is $I_\infty^*=u^{-m_{I_\infty}}I_\infty$.
\item The normalized index of $I_\infty$ is $\delta_{I_\infty}=\deg_u\,[I_\infty^*:A_\infty[y]]$. 
\item The normalized exponent of $I_\infty$  is  $\exp(I_\infty):=\min\{\deg_u(a)\mid a\in A_\infty, aI^*\subset A_\infty[y]\}$.
\end{itemize}
\end{definition}

\noindent

\begin{lem}\label{lem:index exp infinity} We have  $A_\infty[y]\subset B_\infty \subset I_\infty^*$. In particular, $ 0\le \delta_{B_\infty} \le \delta_{I_\infty}$ and $0\le \exp(B_\infty)\le \exp(I_\infty)$. Moreover,
$$
\exp(I_\infty)\le \delta_{I_\infty}\le n\exp(I_\infty).
$$
\end{lem}

\begin{proof}
Similar to the proof of Lemma \ref{lem:index exp}, except that we work over the ring $A_\infty=k[u]_{(u)}$. 
\end{proof}

We denote for short $\delta_\infty=\delta_{B_\infty}$. Analogously to Theorem \ref{thm:complexity_integral_basis}, we deduce from \cite{PoWe24} :

\begin{thm} \label{thm:integral_basis_infinity}
    Let $I_\infty$ be a fractional ideal of $B_\infty$ given by OM-representation. There is a deterministic algorithm which computes a triangular $A_\infty$-basis of $I_\infty$ with :
    \begin{itemize}
        \item $\O_\epsilon(n\delta_{I_\infty})$ operations in $k$ if $char(k)=0$ or $char(k)> n$,
        \item $\O_\epsilon(n\delta_{I_\infty} +\delta_\infty^2)$ operations in $k$ if $char(k)\leq n $.
    \end{itemize}
\end{thm}



\medskip
\noindent
The returned basis has shape $\B_\infty=u^{m}\B_{\infty}^*$ with $m=m_{I_\infty}\in \Z$ as in Definition \ref{def:height infinity} and where
\begin{equation}\label{eq:triang_basis_infinity}
\B_{\infty}^*=\left(1, \frac{h_1(u,y)}{u^{m_1}},\ldots , \frac{h_{n-1}(u,y)}{u^{m_{n-1}}}\right)
\end{equation}
is an $A_\infty$-basis of the normalized ideal $I_\infty^*$ which satisfies:

(i) $0\le m_1 \le m_2 \cdots \le m_{n-1}$. 

(ii) $ h_i\in k[u][y]$ is monic of degree $i$ in $y$ and $\deg_u(h_i)< m_i$. 

\medskip
\noindent
The $m_i$'s are uniquely determined. We have 
 $\delta_{I_\infty}=m_1+\cdots +m_{n-1}$ and  $\exp(I_\infty)=m_{n-1}$. The complexity is again quasi-linear with respect to the size of the output.

\subsection{Various bounds for the exponents}

We will need  upper bounds for the  exponents $\exp(I)$ and $\exp(I_\infty)$  in terms of the divisor $D$.  A key point is the following explicit formula. 

\begin{lem}\label{lem:formula_d0}
Let $I=\prod_{\p\in \P_0} \p^{-n_\p}$ and $I_\infty=\prod_{\p\in \P_\infty} \p^{-n_\p}$. For $p\in \P_A\cup\{\infty\}$, we let $m_p:=\min\left\{\left \lfloor\frac{n_\p}{e_\p}\right \rfloor \mid \p|p\right\}$. With notations of Definition \ref{def:height} and Definition \ref{def:height infinity} :
\begin{enumerate}
\item We have 
$
q_I=\prod_{p\in \P_A} p^{-m_p}$ and $I^*=\prod_{p\in \P_A}\prod_{\p|p}\p^{e_\p m_p-n_\p}$ 
\item We have
$m_{I_\infty}=-m_\infty$ and $I_\infty^*=\prod_{\p|\infty}\p^{e_\p m_\infty-n_\p}$. 
\end{enumerate}
\end{lem}

\begin{proof}
This follows from \cite[Lem.2]{PoWe24}.
\end{proof}

For $q=a/b\in K=k(t)$ with $a,b\in k[t]$ coprime, we denote $h(q)=\deg(a)+\deg(b)$ the height of $q$. 

\begin{prop} \label{prop:bound for expI}
    Let $D\in \Div(L)$ and let $I=I(D)$ and $I_\infty=I_\infty(D)$. We have 
    \begin{equation}\label{eq:rough_bound}
    \exp(B)+\exp(B_\infty) \,\, \le \,\, \exp(I)+\exp(I_\infty) \,\, \leq \,\, \deg(D^+)+ \deg(D^-)+ \exp(B)+\exp(B_\infty)
    \end{equation}
and $h(q_I)+|m_{I_\infty}|\le \deg(D^+)+\deg(D^-)$. 
\end{prop}

\begin{proof}
We first prove the second point. We have $h(q_I)+|m_{I_\infty}|=\sum_{p\in \P_K} |m_p|\deg p$ from Lemma \ref{lem:formula_d0} (with convention $\deg \infty=1$). Let $p\in \P$ and let $\p_0|p$ such that $m_p=\left \lfloor\frac{n_{\p_0}}{e_{\p_0}}\right \rfloor$. 

$\bullet$ If $n_{\p_0} > 0$ then $0<n_{\p_0}/e_{\p_0}\le n_{\p_0}$.  Hence, $0\le \lfloor n_{\p_0}/e_{\p_0} \rfloor\le n_{\p_0}$ and $|m_p|\le |n_{\p_0}|$.

$\bullet$ If $n_{\p_0} < 0$, then $0>n_{\p_0}/e_{\p_0} \ge n_{\p_0}$. Hence, $0\ge \lfloor n_{\p_0}/e_{\p_0} \rfloor \ge n_{\p_0}$ and $|m_p|\le |n_{\p_0}|$. 

\noindent As $\deg p \le \deg \p_0$ by Proposition \ref{prop:degree_place}, we get $|m_p|\deg p\le |n_{\p_0}|\deg \p_0\le \sum_{\p|p} |n_{\p}|\deg \p$. Summing over all $p\in \P_K$, we get
$$
h(q_I)+|m_{I_\infty}|\le \sum_{\p\in \P_K} |n_{\p}|\deg \p = \deg(D^+)+\deg(D^-),
$$
as required. 

We prove now inequality \eqref{eq:rough_bound}. Since $A[x]\subset B\subset I^*$, clearly $\exp(B)\le \exp(I)$. Let $a,b\in A$ such that $aI^*\subset B$ and  $bB\subset A[x]$. We have $abI^*\subset A[x]$. Hence, $\exp(I)\leq \deg(a)+\deg(b)$. Considering $b$ of minimal degree, we obtain $\exp(I)\leq \deg(a)+\exp(B)$ and we are left to bound $\deg a$ in terms of $D$. Let us first construct such an $a$. Denote $l_p=\max\left\{\lceil n_\p/e_\p\rceil\mid \p|p\right\}$ for $p \in \P_A$ and consider $a=\prod_{p\in \P_A} p^{l_p-m_p}$. Since $l_p\ge m_p$, we have $a\in A$. For all $p\in \P_A$ and all $\p|p$, we have $v_\p(p)=e_\p$ (Definition \ref{def:residual_degree}) and Lemma \ref{lem:formula_d0} gives
$$
v_\p(a I^*)=e_\p (l_p-m_p) + e_\p m_p -n_\p=e_\p l_p -n_\p \ge 0.
$$
It follows that $aI^*\subset B$ and we are left to bound $\deg a$. We have 
$$
l_p\le \max_{\p|p, n_\p>0} \left \lceil \frac{n_\p}{e_\p}\right \rceil \le \max_{\p|p, n_\p>0} n_\p \qquad {\rm and}\qquad m_p\ge \min_{\p|p, n_\p<0} \left \lfloor \frac{n_\p}{e_\p}\right \rfloor \ge \min_{\p|p, n_\p<0} n_\p.
$$
Using again $\deg p\le \deg \p$ for all $\p|p$, we get
\begin{align*}
        \deg(a)=\sum_{p\in \P_A}(l_p-m_p)\deg(p)   &\leq \sum_{\p \in  \P_0\mid n_\p>0} n_\p \deg \p - \sum_{\p \in  \P_0\mid n_\p<0} n_\p \deg \p=\deg(D_0^+)+\deg(D_0^-)
    \end{align*}
    where $D_0$ is the finite part of $D$. 
    Hence  $\exp(I)\leq \deg(D_0^+)+\deg(D_0^-)+\exp(B)$. We show in the same way that $\exp(I_\infty)\leq \deg(D_1^+)+\deg(D_1^-)+\exp(B_\infty)$ where $D_1$ is the infinite part of $D$. As $D=D_0+D_1$, the claim follows. 
\end{proof}

%

\medskip
\noindent

The upper bounds of Proposition \ref{prop:bound for expI} are sharp in some cases, but quite bad in some other cases. Let us provide another estimate in terms of the divisor defined by the  normalized ideals.

\begin{prop}\label{prop:expI_vs_degD}
Let $D^*$ be the divisor defined by the pair of fractional ideals $(I^*,I_\infty^*)$ following Proposition \ref{prop:LD=IcapIinfty}. \begin{enumerate}  
\item We have $D^*\ge 0$ and $D^*=D+\divi(q)+ r D_\infty$ for some $q\in K=k(t)$ and some $r\in \Z$.
\item We have $q=q_I$ and $r=m_{I_\infty}+\deg q_I$.
\item If $D_0^*\ne D^*$ satisfies point 1 for some $q_0,r_0$, then $D_0^*> D^*$, $r< r_0$ and $q_0/q\in k[t]$, non constant. 
\item The following inequality holds :
\begin{equation}\label{eq:finner_bound}
\exp(I)+\exp(I_\infty)\le \deg (D^*)+\exp(B)+\exp(B_\infty)
\end{equation} 
\end{enumerate}
\end{prop}

\begin{proof}
Point 1. Since $B\subset I^*$ and $B_\infty \subset I^*_\infty$ by respectively Lemma \ref{lem:index exp} and Lemma \ref{lem:index exp infinity}, each $\p$ appears with negative exponents in $I^*$ and $I^*_\infty$, and we deduce from \eqref{eq:ideaux} that $D^*\ge 0$.
Let us consider $q=q_I$ and $m=m_{I_\infty}$. We have $I(D)=I=q I^*=q I(D^*)=I(D^*-\divi(q))$, the second equality by Definition \ref{def:height}, the third equality by definition of $D^*$ and the fourth equality by \eqref{eq:ideaux}. Hence, 
\begin{equation}\label{eq:Dprime}
D=D^*-\divi(q)+ D'
\end{equation}
for some $D'$ supported at infinity by \eqref{eq:ideaux}. In the same way, $I_\infty =u^m I_\infty^*=t^{-m}I_\infty^*$ (Definition \ref{def:height infinity}) implies that $I_\infty(D)=I_\infty(D^*+\divi(t^m))$, hence
\begin{equation}\label{eq:Dseconde}
D=D^*+\divi(t^m) + D^{''}
\end{equation}
for some $D^{''}$ supported at finite places. Let $D_\infty=\divi_\infty(t)=\divi_0(t^{-1})\ge 0$ be the divisor at infinity.  Since $q\in K$, the infinite part of $-\divi(q)$ is exactly $\deg(q) D_\infty$ while the infinite part of $\divi(t^m)$ is $-m D_\infty$. By equating the parts at infinity in \eqref{eq:Dprime} and \eqref{eq:Dseconde}, we deduce that $D'=-(m+\deg q) D_\infty$. Using \eqref{eq:Dprime} again, we get $D^*=D+\divi(q)+r D_\infty$, with $r=m+\deg q$. This shows point 1 together with the fact that $q=q_I$ and $m=m_{I_\infty}$ are solutions. We show points 2 and 3 simultaneously. We keep notations $q=q_I$ and $m=m_{I_\infty}$ and $r=m+\deg q$.  Let $D_0^*=D+ r_0 D_\infty+\divi(q_0)$ with $D^*_0\ge 0$, $q_0\in K$ and $m_0\in \Z$. Thus
\begin{equation}\label{eq:LDr0}
q_0\in  \L(D+ r_0 D_\infty)\cap K.
\end{equation}
By Proposition \ref{prop:LD=IcapIinfty}, we get in particular
$q_0\in  I(D+ r_0 D_\infty)\cap K=I(D)\cap K=q A,$
the first equality by \eqref{eq:ideaux} since $r_0 D_\infty$ is supported at infinity, and the second equality by Definition \ref{def:height} of $q=q_I$. Thus $q_0=aq$ for some $a\in A=k[t]$. In particular, $a$ has no poles at finite places and $\divi(q_0)-\divi(q)=\divi(a)=\divi_0(a)-\deg(a)D_\infty$. We get 
\begin{equation}\label{eq:D0*vsD*}
D=D_0^*-r_0 D_\infty-\divi(q_0) = D^*-r D_\infty-\divi(q) \quad \Longrightarrow \quad D_0^*=D^*+\divi_0(a)+(r_0-r-\deg a) D_\infty. 
\end{equation}
On the other hand, \eqref{eq:LDr0} and Proposition \ref{prop:LD=IcapIinfty} give $q_0 \in I_\infty(D+ r_0 D_\infty)\cap K$. Since $r_0 D_\infty$ is the infinite part of $\divi(t^{-r_0})$, we get
$$
I_\infty(D+r_0 D_\infty)\cap K= I_\infty(D+\divi(t^{-r_0}))\cap K= t^{r_0} I_\infty(D)\cap K=t^{r_0-m}A_\infty,
$$
the last equality by Definition \ref{def:height infinity} since $m=m_{I_\infty}$. Hence $q_0 t^{m-r_0}\in A_\infty$, that is $\deg q_0+m-r_0\le 0$. As $m=r-\deg(q)$, we get $r_0-r-\deg a \ge 0$. 
Combined with \eqref{eq:D0*vsD*}, we deduce that $D_0^*\ge D^*$, with equality if and only if $a\in k$ and $r=r_0$. This proves points 2 and 3. Finally, we have  $\exp(I)=\exp(I^*)$ by Definition \ref{def:height} and $I(D^*)=I^*$ by definition of $D^*$. Hence, $\exp(I)=\exp(I(D^*))$ and similarly at infinity. Point 4 thus follows from Proposition \ref{prop:bound for expI} applied to the effective divisor $D^*$.
\end{proof}

\medskip
\noindent
In short, Proposition \ref{prop:expI_vs_degD} says that :

$\bullet$ The divisor $D^*$ is the effective divisor of minimal degree which is $K$-linearly equivalent to $D$ modulo $D_\infty$.

$\bullet$ The quantity 
$\exp(I)+\exp(I_\infty)$ measures the defect of $K$-principality of $D$ modulo $D_\infty$. 

$\bullet$ The quantity $r_D=\deg(q_I)+m_{I_\infty}$ obeys to the following relation:
$$
\deg D^*=\deg D + n r_D \qquad {\rm and}\qquad r_D=\min\left\{r\in \Z, \,\, \L(D+r D_\infty)\cap K\ne \{0\}\right\},
$$

first equality by point 1 since $\deg(\divi(q))=0$ and $\deg D_\infty=n$ and second equality by combining points 1, 2, 3.

\begin{rmq}\label{rem:expI_vs_degD}
The bound \eqref{eq:finner_bound} is possibly much better than \eqref{eq:rough_bound}. 
In particular, if $D=\divi(q)+r D_\infty$ for some $q\in K$ and some $r\in \Z$, then $D^*=0$ and \eqref{eq:finner_bound} gives the minimal possible value $\exp(I)+\exp(I_\infty)=\exp(B)+\exp(B_\infty)$ which does not depend on $\deg(D)$, while the upper bound \eqref{eq:rough_bound} can be arbitrarily large when $h(q)$ increases. 

On the other hand, if $D$ is negative and far from being $K$-principal modulo $D_\infty$, the bound \eqref{eq:rough_bound} is better. An example is $D= -N \p$ with $N >0$, $e_\p=f_\p=1$. We get $\deg D=-N$, $r_D=\deg q_I+m_{I_\infty}=N$ by Lemma \ref{lem:formula_d0}, hence $\deg D^* = \deg D + n r_D = (n-1)N> N=\deg D^++\deg D^-$. 
\end{rmq}

\section{From integral bases to Riemann--Roch spaces}\label{sec:5}

\medskip
Given some integral bases of $I(D)$ and $I_\infty(D)$, we want now to compute a $k$-basis of $\L(D)=I(D)\cap I_\infty(D)$. We follow the algorithm of Hess \cite{hess} (see also Bauch's thesis \cite{bauch}), in the vein of the pioneer work of Schmidt \cite{schmidt}. Recall that the divisor at infinity is $D_\infty=\divi_\infty(t)$. It is supported at infinite places. We will compute a $k$-basis of $\L(D+r D_\infty)$ for any integer $r$.

\subsection{Hess's algorithm}

A key point is the notion of reduced basis, that we will rather express in terms of row reduced matrices. For $b\in k(t)^n$, we let $\deg(b)$ be the maximal degree of the entries of $b$.

\begin{definition} \label{def:reduced} Let $b_0,\ldots,b_{n-1}\in k(t)^n$. We denote  $\Mb=(b_0,\ldots,b_{n-1})\in k(t)^{n\times n}$ the matrix with rows $b_i$.
\begin{itemize}
\item The row-degree of $\Mb$ is $\rdeg(\Mb)=(\deg(b_0),\ldots,\deg(b_{n-1}))\in \Z^n$.
\item We say that $\Mb$ is row reduced if $$\deg\left(\sum \lambda_i b_i\right)=\max_i \deg \left(\lambda_i b_i\right) \,\,\,\, \text{for  all} \,\,\, \lambda_i\in k(t).$$ 
\end{itemize}
\end{definition}

\begin{lem}\label{lem:row_reduced} 
For $s\in \Z^n$, denote $|s|\in \Z$ the sum of the entries of $s$. Let $\Mb\in k(t)^{n\times n}$ be non singular.
\begin{enumerate}
\item The matrix $\Mb$ is row reduced if and only if $\deg \det(\Mb)=|\rdeg(\Mb)|$.
\item  There exists $\bold{U}\in k[t]^{n\times n}$ unimodular such that $\bold{U}\bold{M}$ is row reduced. 
\end{enumerate}
\end{lem}

\begin{proof}
This is standard, see e.g. \cite[Ch.6.3]{Kailath}, \cite[Sec.2.7]{zhou} or \cite[Lem.4.1]{hess}.
\end{proof} 

As in the previous section, we assume that $L=k(t)[X]/(f)$ for some irreducible monic separable polynomial $f$ of degree $n$. Thus $L=k(t)[x]$ is a $k(t)$-vector space of dimension $n$. Here is the key result of \cite{hess}.

\begin{prop}\label{prop:hess}
Let $D\in \Div(L)$. Let $\B=(b_0,\ldots,b_{n-1})$ be an $A$-basis of $I(D)$ and $\B_\infty$ an  $A_\infty$-basis of $I_\infty(D)$. 
\begin{itemize}
\item Both $\B$ and $\B_\infty$ form a basis of $L$ as a $k(t)$-vector space. 
\item Let $\Pb\in k(t)^{n\times n}$ be the matrix of the basis $\B$ expressed in the basis $\B_\infty$. If $\Pb$ is row reduced, then for all $r\in \Z$, the Riemann--Roch space $\L(D+rD_\infty)$ has $k$-basis 
$$(b_i t^{j}, \, 0\leq i\leq n-1, \, 0\leq j\leq d_i+r)$$
where $(d_0,\ldots,d_{n-1})=-\rdeg(\Pb)$. 
\item The integers $d_i\in \Z$ are uniquely determined (up to permutation) by the subring $A\subset L$ and the divisor $D$. 
\end{itemize}
\end{prop}


\begin{proof}
This follows from Theorem 5.1 and Algorithm 6.1 in \cite{hess} (with the minor difference that we deal here with row vectors while \cite{hess} deals with column vectors). 
The theorem arises from the existence of bases $(b_0,\ldots, b_{n-1})$ of $I=I(D)$ and $(v_0,\ldots,v_{n-1})$ of $I_\infty=I_\infty(D)$ such that $b_i=t^{-d_i}v_i$ holds with unique rational integers $d_0\ge \cdots \ge d_{n-1}$. 
Considering thoses bases and Proposition \ref{prop:LD=IcapIinfty}, we easily understand the obtainment of such $k$-basis of $\L(D)$ by writing element in $L$ in both bases. 
The condition on the $b_i$ arises from the way to diagonalize $\Pb$ a transition matrix from a basis of $I $ to a basis of $I_\infty$. 
We   row-reduce $\Pb$ which translates into a change-of-basis on $I$  and we column-reduce  $\Pb$ which translates into a change-of-basis of $I_\infty$. 
The matrix obtained is diagonal. In our case, we are only searching for the family $(b_0,\ldots,b_{n-1})$ so the row-reduction suffices.
\end{proof}

\begin{definition} \label{def:compressed_basis}
A family $((b_0,d_0),\ldots,(b_{n-1},d_{n-1}))$ as in Proposition \ref{prop:hess} is called a compressed basis of $\L(D)$. 
\end{definition}

\begin{rmq}\label{rem:compressed_basis}
Note that a compressed basis exists even if $\L(D)=\{0\}$, which corresponds to the case $d_i<0$ for all $i$. Computing a compressed basis in such a case makes sense if we want to compute $\L(D+rD_\infty)$ for some $r\in \Z$. 
\end{rmq}

\medskip
\noindent
In what follows, we freely identify $b\in L=k(t)[x]$ with the row vector of its coefficients in the basis $1,x,\ldots,x^{n-1}$ of the $k(t)$-vector space $L$. We are led to the following algorithm :

\begin{algorithm}[H]
\caption{Hess's algorithm }\label{alg:Hess}
\begin{algorithmic}[1]
\Require 
$L=k(t)[X]/(f)$ with $f\in k[t][X]$ monic and separable, and $D\in \Div(L)$. 
\Ensure A compressed basis of $\L(D)$. 
\State Compute $\Mb$ with rows an $A$-basis $\B$ of $I(D)$.
\State Compute $\Nb$ with rows an $A_\infty$-basis $\B_\infty$ of $I_\infty(D)$.
\State Compute $\bold{U}\in k[t]^{n\times n}$ unimodular such that $\Ub \Mb \Nb^{-1}$ is row reduced.
\State Compute $\bold{U}\Mb$, say $\bold{U}\Mb=(b_0,\ldots,  b_{n-1})$.
\State \Return $((b_i,d_i), \, 0\leq i\leq n-1)$
where $(d_0,\ldots,d_{n-1})=-\rdeg(\Ub \Mb \Nb^{-1})$.
\end{algorithmic}
\end{algorithm}

\begin{thm}\label{thm:hess}
This algorithm is correct. 
\end{thm}

\begin{proof}
The existence of $\Ub$ is ensured by Lemma \ref{lem:row_reduced}. Since  $\bold{U}\in k[t]^{n\times n}=A^{n\times n}$ is unimodular, the rows of $\Ub \Mb$ still form an $A$-basis $\B'$ of $I(D)$. The matrix of $\B'$ in the basis $\B_\infty$ is given by $\Pb=\Ub \Mb \Nb^{-1}$, which is row reduced. The result thus follows from Proposition \ref{prop:hess}. 
\end{proof}

Complexity of steps 1 and 2 have been analyzed in the previous section. We need to study the complexity of the remaining steps. 
A delicate point is that the computation of $\Mb \Nb^{-1}$ and $\Ub \Mb$ is too expensive if we consider the matrix $\Nb$ of the triangular $A_\infty$-basis of $I_\infty(D)$ given by  \eqref{eq:triang_basis_infinity}. We will need to compute first a suitable normal form of $\Nb$.  Let us first state some auxiliary complexity results about polynomial matrices.

\subsection{Multiply, invert and reduce polynomial matrices}

In what follows, the degree of a matrix is the maximal degree of its entries.

\begin{prop}\label{prop:multiplication}
Let $\Eb,\Fb\in k[t]^{n\times n}$ and let $d=\deg \Eb$ and $e=\deg \Fb$.  Then :
\begin{enumerate}
\item $\deg(\Eb \Fb)\le d+e$.
\item We can compute $\Eb \Fb$ in time $\O_\epsilon(n^\omega(d+e+1))$.
\end{enumerate}
\end{prop}

\begin{proof}
By considering $\Eb,\Fb$ as polynomials in $t$ with coefficients in $k^{n\times n}$, the inequality $\deg(\Eb \Fb)\le d+e$ is clear. The second point follows from \cite{CaKa90}.
\end{proof}

Next, we need to estimate the cost of computing a row reduced form (Definition \ref{def:reduced}) of a polynomial matrix.

\begin{prop}\label{prop:reduce}
Let $\Eb\in k[t]^{n\times n}$ of degree $d$. We can compute a row reduced form $\Eb_{\mathrm{red}}$ of $\Eb$ in time $\Ot(n^{\omega}(d+1))$. We have then $\deg \Eb_{\mathrm{red}}\le \deg \Eb$.
\end{prop}

\begin{proof}
First point follows for instance from \cite[Thm 1.4]{neiger-vu}. We have $\overline{\rdeg(\Eb)}\geq \overline{\rdeg( \Eb_{\mathrm{red}})}$, see \cite[Lem 2.14]{zhou}, where $\overline{a}$ is the list of the entries of $a\in \Z^{n}$ sorted in increasing order. Particularly, it implies $\deg(\Eb_{\mathrm{red}})\leq \deg(\Eb)$.
\end{proof}

Finally, we will use the following key fact concerning the inversion of a polynomial matrix :

\begin{prop}\label{prop:inverse}
Let $\Fb\in k[t]^{n\times n}$ be a row reduced matrix and let $d\in \N$ such that :
$$\deg \Fb\le d \qquad {\rm and}\qquad t^d \Fb^{-1} \in k[t]^{n\times n}.$$
Then $\deg (t^d \Fb^{-1})\le d$ and $\Fb^{-1}$ can be computed in time $\O_\epsilon(n^\omega d)$. 
\end{prop}

\begin{proof}
We follow \cite{rosenkield}.  
Using the determinant formula, the cofactor matrix of $\Fb$  satisfies $\deg({}^t\com(\Fb))\leq |\rdeg(\Fb)|$. Since $\Fb$ is row reduced, we have $\deg(\det(\Fb))= |\rdeg(\Fb)|$ (Lemma \ref{lem:row_reduced}). Using $\Fb^{-1}={}^t\com(\Fb)/\det(\Fb)$, we obtain $\deg(\Fb^{-1})\leq 0$. 
Hence $t^d \Fb^{-1}$ has degree at most $d$. The cost estimation follows by combining \cite[Lemma 3.8]{rosenkield} and \cite[Lemma 3.6]{rosenkield}  (first condition is expressed in \cite{rosenkield} in terms of column degree, but this has no matter for our purpose up to transpose). The strategy is as follows : we compute $\Fb^{-1}$ modulo $(t-1)$ and use fast high order lifting \cite{stor13} to deduce $\Fb^{-1}$  modulo $(t-1)^{d+1}$. We deduce $t^d \Fb^{-1}$  modulo $(t-1)^{d+1}$ and we use fast radix conversion to deduce the $t$-adic expansion of $t^{d} \Fb^{-1}$.  
\end{proof}

\begin{rmq}
There exist finner complexity results about multiplication and reduction of polynomials matrices,  which take into account the average degree of the rows (or columns) of the input matrices (see e.g. \cite{neiger-vu}). We don't need these improvements for our purpose. 
\end{rmq}

\subsection{Revisiting Hess's algorithm} 

From now on, we denote $I=I(D)$ and $I_\infty=I_\infty(D)$ and we let $\B,\B^*$ and $\B_\infty,\B_\infty^*$ the triangular basis respectively defined by \eqref{eq:triang_basis} and \eqref{eq:triang_basis_infinity}. 
We denote for short
\begin{equation}\label{eq:ed}
d=\exp(I)\qquad \mathrm{and}\qquad e=\exp(I_\infty).
\end{equation}
(see Definition \ref{def:height} and Definition \ref{def:height infinity}). We thus have $d=\deg p_{n-1}$ and $e=m_{n-1}$ with notations \eqref{eq:triang_basis} and \eqref{eq:triang_basis_infinity}.

\begin{lem}\label{lem:Mtilde}
Let $\Mb^*$ be the matrix of the basis $\B^*$ expressed in the basis $1,x,\ldots,x^{n-1}$. The matrix $\Mb^*$ is lower triangular and we have
$$
p_{n-1}\Mb^*=\Mbt,\quad  \Mbt\in k[t]^{n\times n},\quad \deg(\Mbt)= d.
$$
\end{lem}

\begin{proof}
This follows straightforwardly from \eqref{eq:triang_basis}  and \eqref{eq:ed}.
\end{proof}

Recall that we denote $u=t^{-1}$. Hence $\deg_u=-\deg_t$. 

\begin{lem}\label{lem:Ntilde}
Let $\Nb^*$ be the matrix of the basis $\B_{\infty}^*$ expressed in the basis $1,x,\ldots,x^{n-1}$ and let $\lambda$ as in \eqref{eq:lambda}. We have
$$
u^{e}\Nb^*=\Nbt \quad {\text where} \quad  \Nbt\in k[u]^{n\times n},\quad \deg_u(\Nbt)\le e+n\lambda.$$
Moreover, we have $u^{e+n\lambda} \Nbt^{-1}\in k[u]^{n\times n}$. 
\end{lem}

\begin{proof}
Let $\Nb'$ be the matrix of $\B_{\infty}^*$ in the basis $1,y,\ldots,y^{n-1}$. As $y=x u^{\lambda}$, the matrix of $\B_{\infty}^*$ in the basis $1,x,\ldots,x^{n-1}$ is  $\Nb^*=\Nb' \Db$ where $\Db=\diag(1,u^\lambda,\ldots,u^{(n-1)\lambda})$.  It follows straightforwardly from \eqref{eq:triang_basis_infinity} and \eqref{eq:ed} that $u^e\Nb'\in k[u]^{n\times n}$, $\deg_u(u^e \Nb')=e$. Hence, $\Nbt=u^e \Nb' \Db\in k[u]^{n\times n}$ and 
$\deg_u \Nbt\le \deg (u^e\Nb')+\deg \Db\le e+ n\lambda$, as required. For the second point, we have $u^{e+n\lambda}\Nbt^{-1}=u^{n\lambda}\Db^{-1}\Nb'^{-1}$. Since $u^{n\lambda} \Db^{-1}\in k[u]^{n\times n}$, we are reduced to show that $\Nb'^{-1}\in k[u]^{n\times n}$. We deduce from \eqref{eq:triang_basis_infinity} that $\Nb'=\Db' \Tb$ where $\Tb\in k[u]^{n\times n}$ is lower triangular with ones on the diagonal (in particular, $\Tb$ is unimodular) and where 
$\Db'$ is a diagonal matrix with entries in $k[u^{-1}]$. Thus both $\Tb^{-1}$ and $\Db'^{-1}$ lie in $k[u]^{n\times n}$ and  $\Nb'^{-1}\in k[u]^{n\times n}$ as required.
\end{proof}

\begin{lem}\label{lem:Ntilde0}
Let $\Nbt_{\mathrm{red}}\in k[u]^{n\times n}$ be a row reduced form of $\Nbt$. We have :
\begin{enumerate}
\item [(i)] $\deg_u(\Nbt_{\mathrm{red}})\le e+n\lambda$

\item[(ii)] $u^{e+n\lambda} \Nbt_{\mathrm{red}}^{-1}\in k[u]^{n\times n}$

\item[(iii)] $\deg_u  (u^{e+n\lambda} \Nbt_{\mathrm{red}}^{-1})\le e+n\lambda$. 
\end{enumerate}
\noindent
Given $\Nbt$, we can compute both the matrix $\Nbt_{\mathrm{red}}$  and its inverse $\Nbt_{\mathrm{red}}^{-1}$ in time $\Ot(n^\omega (e+n\lambda))$.
\end{lem}

\begin{proof}
Item (i) and the complexity of computing $\Nbt_{\mathrm{red}}$ follow from Lemma \ref{lem:Ntilde} and Proposition \ref{prop:reduce}. We have $\Nbt_{\mathrm{red}}=\Ub\Nbt$ with $\Ub\in k[u]^{n\times n}$ unimodular, thus $\Nbt_{\mathrm{red}}^{-1}=\Nbt^{-1}\Ub^{-1}$ with $\Ub^{-1}\in k[u]^{n\times n}$ and  (ii) follows from Lemma \ref{lem:Ntilde}. Finally (iii) and the complexity of computing $\Nbt_{\mathrm{red}}^{-1}$ follow from (i), (ii) and Proposition \ref{prop:inverse}. 
\end{proof}

In the following two lemmas, we come back to the coordinate $t=u^{-1}$.

\begin{lem}\label{lem:Pb}
Let $\Pb=\Mbt \Nbt_{\mathrm{red}}^{-1}$. Then $\Pb\in k[t]^{n\times n}$ and $\deg_t \Pb\le d+e+n\lambda$. We can compute $\Pb$  in time $\Ot(n^\omega (d+e+n\lambda))$. 
\end{lem}

\begin{proof}
By Lemma \ref{lem:Mtilde}, we have $\Mbt\in k[t]^{n\times n}$, $\deg_t \Mbt=d$. From Lemma \ref{lem:Ntilde0} (ii) and (iii) we deduce that  $\Nbt_{\mathrm{red}}^{-1}\in k[u^{-1}]^{n\times n}=k[t]^{n\times n}$ and $\deg_t \Nbt_{\mathrm{red}}^{-1}\le e+n\lambda$. The result thus follows from Proposition \ref{prop:multiplication}.
\end{proof}

\begin{lem}\label{lem:Pb0}
We can compute a row reduced form $\Pb_{\mathrm{red}}\in k[t]^{n\times n}$ of $\Pb$   in time $\Ot(n^\omega (d+e+n\lambda))$. We have $\deg_t(\Pb_{\mathrm{red}})\le d+e+n\lambda$. 
\end{lem}

\begin{proof}
This follows from Proposition \ref{prop:reduce} and Lemma \ref{lem:Pb}.
\end{proof}

\begin{lem}\label{lem:Mt0}
Let $\Ub\in k[t]^{n\times n}$ be the unimodular matrix such that $\Pb_{\mathrm{red}}=\Ub \Pb$. 
We can compute $\Mbt_{\mathrm{red}}=\Ub \Mbt\in k[t]^{n\times n}$  in time $\Ot(n^\omega (d+e+n\lambda))$.
\end{lem}

\begin{proof} By definition of $\Pb$ in Lemma \ref{lem:Pb}, we have $\Ub \Mbt=\Ub \Pb \Nbt_{\mathrm{red}}= \Pb_{\mathrm{red}} \Nbt_{\mathrm{red}}.$  
By Lemma \ref{lem:Ntilde0} (i) and since $u=t^{-1}$, we deduce that $t^{e+n\lambda}\Nbt_{\mathrm{red}}\in k[t]^{n\times n}$ has degree at most $e+n\lambda$ in $t$. Hence computing  $t^{e+n\lambda}\Pb_{\mathrm{red}}\Nbt_{\mathrm{red}}$ fits in the aimed bound by Lemma \ref{lem:Pb0} and Proposition \ref{prop:multiplication}. There remains to divide the resulting matrix by  $t^{e+n\lambda}$, which can be done for free. 
\end{proof}

\begin{rmq}
Clearly, we could have studied more carefully the complexity of the various steps taking into account in particular that $\Nbt_{\mathrm{red}}$ is a left multiple of the diagonal matrix $\Db=(1,u^\lambda,\ldots,u^{(n-1)\lambda})$. This would not improve the overall cost estimate in general, but this leads to a slight improvement when the index at infinity is zero, see Subsection \ref{ssec:trivial_index_infinity}.
\end{rmq}

We can now revisit Hess algorithm to get complexity estimates.

\begin{algorithm}[H]
\caption{Riemann--Roch}\label{alg:RR}
\begin{algorithmic}[1]
\Require 
$L=k(t)[X]/(f)$ with $f\in k[t][X]$ monic and separable, and  $D\in \Div(L)$ given by OM-representation. 
\Ensure A compressed basis of $\L(D)$. 
\State Compute $\Mbt\in k[t]^{n\times n}$ as in Lemma \ref{lem:Mtilde} using Theorem \ref{thm:complexity_integral_basis}. 
\State Compute $\Nbt\in k[u]^{n\times n}$ as in Lemma \ref{lem:Ntilde} using Theorem \ref{thm:integral_basis_infinity}. 
\State Compute a row reduced form $\Nbt_{\mathrm{red}}$ of $\Nbt$ in $k[u]^{n\times n}$ and compute $\Nbt_{\mathrm{red}}^{-1}$ with Lemma \ref{lem:Ntilde0}.
\State Compute $\Pb=\Mbt \Nbt_{\mathrm{red}}^{-1}\in k[t]^{n\times n}$ with Lemma \ref{lem:Pb}.
\State Compute a row reduced form $\Pb_{\mathrm{red}}=\Ub \Pb$ in $k[t]^{n\times n}$ with Lemma \ref{lem:Pb0}. 
\State Compute $\Mbt_{\mathrm{red}}=\Ub \Mbt$ in $k[t]^{n\times n}$ with Lemma \ref{lem:Mt0}.
\State Compute $\Mb_{\mathrm{red}}=\frac{q}{p_{n-1}}\Mbt_{\mathrm{red}} \in k(t)^{n\times n}$, with $q\in k(t)$ and $p_{n-1}\in k[t]$ as in \eqref{eq:triang_basis}.
\State Let $\delta = m_{n-1}-m-\deg (q/p_{n-1})$ with $m_{n-1}\in \N$ and $m\in \Z$ as in \eqref{eq:triang_basis_infinity}. 
\State \Return $$((b_i,d_i),\,\, 0\leq i\leq n-1)$$ where $(b_0,\ldots,b_{n-1})=\Mb_{\mathrm{red}}$ and 
$(d_0,\ldots,d_{n-1})=-\rdeg \Pb_{\mathrm{red}} +(\delta,\ldots,\delta)$.
\end{algorithmic}
\end{algorithm}

Recall that $b_i\in k(t)^n$ is identified with an element of  $k(t)[x]$ at step 9. 
Using notations of Definition \ref{def:height} and Definition \ref{def:height infinity}, we get:
\begin{thm}\label{thm:RR}
This algorithm returns a correct answer. It performs at most 
\begin{itemize}
\item $\Ot(n^{\omega}(\exp(I)+\exp(I_\infty)+n\lambda)+n^2 h(q_I))$ operations in $k$ if $\car(k)=0$ or $\car(k)>n$,
\item $\Ot(n^{\omega}(\exp(I)+\exp(I_\infty)+n\lambda)+n^2h(q_I) +\delta^2 +\delta_\infty^2)$ operations in $k$ otherwise.
\end{itemize}
\end{thm}

\begin{proof}
\emph{Correctness.} Let $\Mb$ and $\Nb$ be the respective matrices of $\B$ and $\B_\infty$ in the canonical basis $1,x,\ldots,x^{n-1}$. With notations \eqref{eq:triang_basis} and \eqref{eq:triang_basis_infinity} and notations of Lemma \ref{lem:Mtilde} and Lemma \ref{lem:Ntilde}, we thus have 
\begin{equation}\label{eq:M_Mstar}
\Mb=q \Mb^* =\frac{q}{p_{n-1}} \Mbt\qquad \mathrm{and}\qquad \Nb= u^m \Nb^* =u^{m-e} \Nbt.
\end{equation}
By construction, $\Nbt_{\mathrm{red}}=\Vb \Nbt$ for some unimodular matrix $\Vb\in k[u]^{n\times n}$. So $\Vb$ is unimodular over $A_\infty=k[u]_u$ and the rows of $\Nb_{\red}:=u^{m-e} \Nbt_{\mathrm{red}}=\Vb \Nb$ still form an $A_\infty$-basis of $\B_\infty$. The matrix $\Pb_{\mathrm{red}}=\Ub \Mbt \Nbt_{\mathrm{red}}^{-1}$ being row reduced at step 5 (with respect to $\deg_t$), we deduce from Definition \ref{def:reduced} that the matrix 
$\Pb_0:=\Ub \Mb \Nb_{\mathrm{red}}^{-1}$ is row reduced too since it is a $K$-multiple of $\Pb_{\mathrm{red}}$. It follows from Algorithm \ref{alg:Hess} and Theorem \ref{thm:hess} that a compressed basis of $\L(D)$ is given by $(b_0,\ldots,b_{n-1}):=\Ub \Mb$ together with $(d_0,\ldots,d_{n-1}):=-\rdeg(\Pb_0)$.  Using \eqref{eq:M_Mstar} together with $e=\exp(I_\infty)=m_{n-1}$ (by \eqref{eq:triang_basis_infinity}) and $u=t^{-1}$, we get
$$
\Ub\Mb=\frac{q}{p_{n-1}} \Ub \Mbt=\Mb_{\mathrm{red}} \qquad {\rm and}\qquad \Pb_0=\frac{q t^{m-m_{n-1}}}{p_{n-1}}\Pb_{\mathrm{red}},
$$
hence $(b_0,\ldots,b_{n-1})=\Mb_{\mathrm{red}}$ and 
$(d_0,\ldots,d_{n-1})=-\rdeg \Pb_{\mathrm{red}} +(\delta,\ldots,\delta)$ with $\delta=m_{n-1}-m-\deg (q/p_{n-1})$, as returned by Algorithm \ref{alg:RR}. 

\emph{Complexity.} We compute $\Mb^*$ and $\Nb^*$ in the aimed bound thanks to  Theorem \ref{thm:complexity_integral_basis} and Theorem \ref{thm:integral_basis_infinity} combined with the inequalities $n\delta_I\le n^2 \exp(I)$  and $n\delta_{I_\infty}\le n^2 \exp(I_\infty)$ by respectively Lemma \ref{lem:index exp} and Lemma \ref{lem:index exp infinity}. Since $\deg p_{n-1} =\exp(I)$ and the $n^2$ entries of $\Mb^*$ have height $\O(\exp(I))$, we compute $\Mbt=p_{n-1} \Mb^*$ in the aimed bound. We compute $\Nbt$ from $\Nb^*$ for free (multiplication by a power of $t$). The complexity of steps 3, 4, 5, 6 fits in the aimed bound from respectively Lemma \ref{lem:Ntilde0}, Lemma \ref{lem:Pb},  Lemma \ref{lem:Pb0} and Lemma \ref{lem:Mt0}. The polynomial matrix $\Mbt_{\mathrm{red}}$ has degree at most $\exp(I)+\exp(I_\infty)+n\lambda$ by (the proof of) Lemma \ref{lem:Mt0}. Since $q=q_I$ and $\deg p_{n-1}=\exp(I)$, we deduce that computing
$\Mb_{\mathrm{red}}$ at step 7 fits in the aimed bound. 
\end{proof}

Again, we could have been more careful when looking at the complexity of the various multiplications of matrices by elements of $K=k(t)$. However, this would not improve the overall cost estimate. 

\begin{rmq}
Instead of performing the product at step 7, we may also return a list of (non expanded) product 
$(q \tilde{b}_0 ,q \tilde{b}_1/p_{1} , \ldots, q \tilde{b}_{n-1}/p_{n-1})$ where the $\tilde{b}_i$ are given by the rows of $\Mbt_{\mathrm{red}}$. In such a case, we do not need the extra term $h(q_I)$ in the complexity estimate. This is particularly relevant if $\exp(I)<< h(q_I)$, that is when  $D$ is closed to be $K$-principal modulo $D_\infty$ (see Remark \ref{rem:expI_vs_degD}).
\end{rmq}

%
%

\subsection{The case of trivial exponent at infinity}\label{ssec:trivial_index_infinity}

If we assume that $\exp(I_\infty)=0$, the annoying extra term $n\lambda$ disappears in the complexity estimate of Theorem \ref{thm:RR} :

\begin{thm} \label{thm:complexity trivial index}
Suppose that $\exp(I_\infty)=0$. We can compute a compressed basis of $\L(D)$ with \begin{itemize}
\item $\Ot(n^{\omega}\exp(I)+n^2 h(q_I))$ operations in $k$ if $\car(k)=0$ or $\car(k)>n$,
\item $\Ot(n^{\omega} \exp(I) + n^2 h(q_I)+\delta^2)$ operations in $k$ otherwise.
\end{itemize} 
\end{thm}

The proof requires to use shifted reduction of matrices instead of the usual  degree reduction.

\begin{definition}
Let $\vec{s}=(s_0,\ldots, s_{n-1})\in \Z^n$ be a shift and $\Eb\in k(t)^{n\times n}$. The matrix $\Eb$ is said to be row $\vec{s}$-reduced if $\Eb\Db$ is row reduced, where $\Db=\diag(t^{s_0},\ldots, t^{s_{n-1}})$. 
\end{definition}

The shift is used as column weights. 

\begin{thm}[\cite{neiger-vu}]\label{thm:neiger}
Let $\Eb\in k[t]^{n\times n}$ of degree $d$. We can compute a row $\vec{s}$-reduced form of $\Eb$ in time $\Tilde{\O}(n^\omega d)$.
\end{thm}

It is remarkable that the complexity estimate does not depend on the shift $\vec{s}$.

\begin{proof}[Proof of Theorem \ref{thm:complexity trivial index}]
The hypothesis $\exp(I_\infty)=0$ is equivalent to $\delta_{I_\infty}=0$ thanks to Lemma \ref{lem:index exp infinity}, that is $\deg [A_\infty[y]:I_\infty^*]=0$ by Definition \ref{def:height infinity}. Since, $A_\infty[y]\subset I_\infty^*$, we get $A_\infty[y]=I_\infty^*$ by Lemma \ref{lem:index}. Hence $\Nbt=\Db=\diag(1,u^\lambda,\ldots,u^{(n-1)\lambda})$ is diagonal (hence reduced), so the algorithm simplifies. The matrix $\Mbt_{\mathrm{red}}$ of step 6 in Algorithm \ref{alg:RR} coincides with the row  $\vec{s}$-reduced form of $\Mbt$ where $\vec{s}=(1,\lambda,\ldots,(n-1)\lambda)$, and we may replace steps 3, 4, 5, 6 by simply computing the row  $\vec{s}$-reduced form of $\Mbt$, with complexity $\Ot(n^{\omega}\exp(I))$ by Theorem \ref{thm:neiger}. Note that $\Nbt_{\mathrm{red}}=\Db$, hence $\Pb_{\mathrm{red}}=\Mbt_{\mathrm{red}} \Db^{-1}$. Thus 
$\rdeg \Pb_{\mathrm{red}}=\rdeg \Mbt_{\mathrm{red}} +(1,\lambda,\ldots,(n-1)\lambda)$ is deduced for free from the row-degree of $\Mbt_{\mathrm{red}}$. 
\end{proof}

\subsection{An example} 

Let us illustrate Algorithm \ref{alg:RR} on a concrete example.
Let $f=X^3-X^2+t^2\in \F_5(t)[X]$ as in Example \ref{ex:OM_representation}. 

\smallskip
\noindent
$\bullet$ Let $p=tA$. We saw in Example \ref{ex:OM_representation} that there are three primes $\p_1,\p_2,\p_3$ lying above $p$, with OM-representation $t_{\p_1}=[t,x-t]$, $t_{\p_2}=[t,x+t]$ and $t_{\p_3}=[t,x-1]$. We have for instance $tB=\p_1 \p_2 \p_3$ and $(x-1)B=\p_3^2$. 

\smallskip
\noindent
$\bullet$ Let $p'=(t-1)A$. There are exactly two primes $\p_1',\p_2'$ lying above $p'$, with OM-representation  $t_{\p_1'}=[t-1,x+3]$ and $t_{\p_2'}=[t-1,x,x^2+x+2]$. We have $(t-1)B=\p'_1 \p'_2$. 

\smallskip
\noindent
$\bullet$ Let $p_\infty=t^{-1}A_\infty$. Following \eqref{eq:lambda}, we have $\lambda=1$, $y=xt^{-1}$ and $y^3-t^{-1} y^2+t^{-1}=0$. There is a unique prime $\p_\infty$ above $p_\infty$, with OM-representation $\p_{\infty}=[t^{-1},y,y^3+t^{-1}]$. We have $t^{-1} B_\infty = \p_\infty^3$ and $y B_\infty =\p_\infty$.

\medskip
\noindent
Let us compute the Riemann--Roch space associated to the divisor $D=2\p_1 +3\p_2 -\p_1'-\p_2'+\p_\infty \in \Div(\F_5(t)[X]/f)$.

\medskip
\noindent
{\bf Step 1. } We have $I=I(D)=\p_1^{-2}\p_2^{-3}\p_1'\p_2'$. Using the computer algebra system \textsf{SageMath}, we find that this  fractional ideal has basis $$\B=(t-1)\B^*,\qquad \mathrm{with}\qquad \B^*=\left ( 1, \frac{x-1}{t^2}, \frac{x^2-(t+1)x+t(1+t-t^2)}{t^4} \right ).$$ 
We have $I^*=\p_1^{-2}\p_2^{-3}$, $q= t-1$, $p_{n-1}=t^4$ and  $$\Mbt=\begin{pmatrix}
t^4 & 0 & 0 \\
-t^2 & t^2 & 0 \\
t(1+t-t^2) & -(t+1) & 1
\end{pmatrix}.$$
{\bf Step 2.} We have $I_\infty=I_\infty(D)=\p_\infty^{-1}$.
This fractional ideal is normalized and \textsf{SageMath} provides a basis $$\B_\infty=\B_\infty^*=\left ( 1,y,t y^2 \right )=\left ( 1,\frac{x}{t},\frac{x^2}{t} \right ).$$
We have $m=0$, $m_{n-1}=1$  and $$\Nbt=\begin{pmatrix}
1/t & 0 & 0 \\
0 & 1/t^2 & 0 \\
0 & 0 & 1/t^2
\end{pmatrix}.$$
{\bf Step 3.} The matrix $\Nbt$ is diagonal, thus row reduced, and its inverse is obvious.\\

\noindent {\bf Step 4.} We compute $$\Pb=\Mbt \times \Nbt^{-1}=\begin{pmatrix}
t^5 & 0 & 0 \\
-t^3 & t^4 & 0 \\
t^2(1+t-t^2) & -t^2(t+1) & t^2
\end{pmatrix}.$$
{\bf Step 5.} The matrix $\Pb$ is not row reduced. Indeed, $\deg(\det(\Pb))=11<|\rdeg(\Pb)|=13$. We compute a row reduced form of $\Pb$ : $$\Pb_{\mathrm{red}}=\begin{pmatrix}
t^3 +t^2 & -2t^3-t^2 & t^2(t+1) \\
-t^3 & t^4 & 0 \\
0 & -t^4 & t^4+t^2
\end{pmatrix}.$$
{\bf Step 6.} We compute the product $\Mbt_{red}=\Pb_{\mathrm{red}}\times \Nbt_{\mathrm{red}}$ :
$$\Mbt_{\mathrm{red}}=\begin{pmatrix}
t^2 +t & -2t-1 & t+1 \\
-t^2 & t^2 & 0 \\
0 & -t^2 & t^2+1
\end{pmatrix}.$$
{\bf Step 7.} We obtain the following elements, corresponding to the lines of $\Mb_{\mathrm{red}}=\frac{q}{p_{n-1}} \Mbt_{\mathrm{red}}$: $$b_0=\frac{t^2 -1}{t^2}+x \frac{(3t-1)(t-1)}{t^4}+x^2\frac{t^2 -1}{t^4},\quad  b_1=\frac{1-t}{t^2}+ x\frac{t-1}{t^2} \quad \text{ and } \quad b_2=x\frac{1-t}{t^2}+ x^2\frac{(t^2+1)(t-1)}{t^4}.$$
{\bf Step 8.} We have $\delta =4$ and $\rdeg(\Pb_{\mathrm{red}})=(3,4,4)$, hence $(d_0,d_1,d_2)=(1,0,0)$.\\

\noindent {\bf Result : } A $k$-basis of $\L(D)$ is $(b_0, tb_0, b_1, b_2)$.

\section{Arithmetic vs Geometry (proof of Theorem \ref{thm:main})}\label{sec:6}

The complexity given in Theorem \ref{thm:main} is stated in terms of the delta invariant of a projective plane curve. To this aim, we need to relate curves and function fields. More details are given in Appendix \ref{sec:curves}.

\subsection{Curves vs function fields}

Let $\C\subset \mathbb{P}^2_k$ be a degree $n$ irreducible projective plane curve defined by a degree $n$ homogeneous polynomial $F\in k[X_0,X_1,X_2]$. We denote $k(\C)$ the field of rational functions of $\C$. The elements of $k(\C)$ are quotients $G/H$ of homogeneous polynomials of the same degrees, with $H\mod F\ne 0$, and modulo the equivalence relation $G/H\sim G'/H'$ if $GH'-G'H=0$ mod $F$. The field $k(\C)$ is a function field over $k$ (and every function field arises in such a way).

\begin{definition}\label{def:divisor_of_curve}
A divisor of $\C$ is a divisor of the function field $k(\C)$ of $\C$. 
\end{definition}

Equivalently, a divisor of $\C$ can be seen as a $\Z$-combination of closed points of the normalization $\overline{\C}$ of $\C$ (Theorem \ref{thm:projective_normalisation}). If $\C$ is smooth and $k$ is algebraically closed, a divisor is simply a $\Z$-combination of points on a curve, the usual way to think a divisor. 

\medskip

Let us assume that $F\ne X_0$. Denoting $t=X_1/X_0$ and $X=X_2/X_0$, the restriction of $\C$ to the affine chart $U_0=\{X_0\ne 0\}\simeq \A_k^2$ can be identified with the affine curve $C\subset \A^2_k$ defined by
$$
C=\Spec k[t,X]/(f)\quad {\rm where} \quad f(t,X)=F(1,t,X)\in k[t,X].
$$
The coordinate ring of $C$ is $k[C]:=k[t,X]/(f)$  and its function field is $k(C)=\Frac k[C]$. With notations  $A=k[t]$, $K=k(t)$ and $x=X\mod f$ used in the previous sections, we thus have
$$
k[C]=A[x] \quad {\rm and}\quad k(C)=K(x).
$$
Dehomogenization induces an isomorphism $k(\C)\simeq k(C)$. It identifies closed points of $\C\cap U_0$ with closed points of $C$  and induces isomorphisms of local rings $\O_{\C,P}\simeq \O_{C,P}$ for such points.  Note that a closed point $P\in C$ is by definition a maximal ideal of $k[C]$ and the local ring of $C$ at $P$ is $\O_{C,P}=k[C]_P$ (see Section \ref{ssec:closed_points}).


\subsection{Bounds for the indices of $B$ and $B_\infty$}

In order to prove Theorem \ref{thm:main}, there remains to relate the indices $\delta=\deg [B:A[x]]$ and $\delta_\infty=\deg [B_\infty:A_\infty[y]]$ of $f$ as defined in Definition \ref{def:height} and Definition \ref{def:height infinity} with the delta-invariant of the input projective plane curve $\C$ (Definition \ref{def:delta-invariant}). We denote $\overline{\O_{\C,P}}$ the integral closure of $\O_{\C,P}$.  

\begin{prop}\label{prop:delta et delta_infty} Suppose $F$ monic and separable in $X_2$. With notations as above, we have:
\begin{enumerate}
\item $\delta=\sum_{P\in \C\cap U_0} \dim_k \overline{\O_{\C,P}}/\O_{\C,P}$. 
\item $\delta_\infty=\sum_{P\in \C\setminus \C\cap U_0} \dim_k \overline{\O_{\C,P}}/\O_{\C,P}$. 
\end{enumerate}
\end{prop}

Note that according to Proposition \ref{prop:regular_point}, $\overline{\O_{\C,P}}=\O_{\C,P}$ if and only if $P$ is a regular point of $\C$, hence the sums hold over the finite set of singular points of $\C$. 

\begin{proof} 
Let us show point 1. Let $B$ be the integral closure of $A$ in $L=k(C)$. As $A\subset A[x]=k[C]$, we have $B\subset \kCbar$. By hypothesis, $f(t,X)=F(1,t,X)$ is monic and separable in $X$. Hence  $x$ is integral over $A$ which forces $A[x]\subset B$, that is $k[C]\subset B$. As $B$ is integrally closed, it follows that $\kCbar\subset B$, hence $\kCbar=B$. In other words, $B$ coincides with the coordinate ring of the normalization $\Cbar$ of $C$. We  thus have
\begin{equation}\label{eq:delta_kC}
\delta:=\deg [B:A[x]]= \dim_k (B/A[x]) =\dim_k (\kCbar/ k[C]).
\end{equation}
A closed point $P\in \C\cap U_0\simeq C$  is by definition a maximal ideal of $k[C]=A[x]$ and we have $\O_{\C,P}\simeq k[C]_P= A[x]_P$ (localization at $P$). As integral closure commutes with localization, we have an isomorphism of semi-local rings $\overline{\O_{\C,P}}\simeq \kCbar_P=B_P$ and we get
\begin{equation}\label{eq:dimOCbar}
\dim_k (\overline{\O_{\C,P}}/\O_{\C,P})=\dim_k (B_P/A[x]_P).
\end{equation}
We now follow the proof of \cite[Thm.4.17]{kunz}. Let $a\in A[x]$ such that $aB\subset A[x]$. Thus $aA[x]$ and $aB$ are both ideals of $A[x]$, and we have an inclusion $a A[x]\subset aB$. The $k$-algebra $A[x]/a A[x]$ being finitely generated and zero-dimensional, it has a finite number of maximal ideals. Thus, the Chinese remainder theorem (see \cite[Cor.D.4]{kunz}) ensures that 
 $$
A[x]/aA[x]=\prod_{P\in \Spec A[x]} A[x]_P / aA[x]_P$$  and $aB/aA[x]$ identifies with $\prod_{P\in \Spec A[x]} aB_P / aA[x]_P.$ Recall that $C=\Spec A[x]$. We obtain
$$
\dim_k (B/A[x]) = \dim_k(aB/aA[x]) = \sum_{P\in C} \dim_k (aB_P / aA[x]_P)=\sum_{P\in C} \dim_k (B_P / A[x]_P)
$$ 
which combined with \eqref{eq:delta_kC} and \eqref{eq:dimOCbar} concludes for the first point.

Let us show point 2.  Since $F$ is monic in $X_2$, we have $F(0:0:1)\ne 0$. Thus $(0:0:1)\notin \C$ and all points of $\C\setminus \C\cap U_0 =\C\cap(X_0=0)$ (the points at infinity) belong to the affine chart $U_1=\{X_1\ne 0\}$. Denoting $u=X_0/X_1=t^{-1}$ and $Y=X_2/X_1=X t^{-1}$, the affine equation of $\C$ in the chart $U_1\simeq \A^2_k$ is given by the polynomial
$$f_\infty\in k[u,Y],\quad f_\infty(u,Y):=u^d f(u^{-1},u^{-1} Y).$$
The residue class $y$ of $Y$ modulo $f_\infty$ satisfies $y=x t^{-1}$ in $L=k(t,x)=k(u,y)$. As $f$ has total degree $n$, the integer $\lambda=\lambda(f)$ defined in \eqref{eq:lambda} satisfies $\lambda=1$  and Definition \ref{def:height infinity} becomes 
$$
\delta_\infty=\deg [B_\infty:A_\infty[y]],
$$ 
where $u=t^{-1}$, $y=x t^{-1}$, $A_\infty=k[u]_{(u)}$ and $B_\infty$ is the integral closure of $A_\infty$ in $L$.  The line at infinity $X_0=0$ has local equation $u=0$ and the points of $\C\setminus \C\cap U_0 =\C\cap(X_0=0)$ are in one-to-one correspondence with the maximal ideals of $k[u,y]$ containing $u$, hence with the maximal ideals $P$ of $A_\infty [y]$. By the same reasoning as above, we get 
$$
\dim_k (B_\infty/A_\infty[y]) = \sum_{P\in \C\setminus \C\cap U_0} \dim_k (B_{\infty,P} / A_\infty[y]_P).$$
and we conclude thanks to the isomorphisms $A_\infty [y]_P\simeq \O_{\C,P}$ and $B_{\infty,P}\simeq \overline{\O_{\C,P}}$. 
\end{proof}

Denote by $\delta(\C)$ the delta invariant of $\C$ and by $g(\C)$ its geometric genus, as defined in Section \ref{ssec:genus_formula}. We get:

\begin{coro}\label{cor:delta_vs_index}
Suppose $F$ monic and separable in $X_2$. Then $\delta+\delta_\infty=\delta(\C)=\frac{(n-1)(n-2)}{2}-g(\C)$.
\end{coro}

\begin{proof}
First equality follows from Proposition \ref{prop:delta et delta_infty} and Definition \ref{def:delta-invariant}. Second equality follows from the genus formula  (Corollary \ref{coro:lien genre delta}).
\end{proof}

\subsection{Proof of Theorem \ref{thm:main}}

Theorem \ref{thm:main} follows immediately from the slightly more general result. 

\begin{thm}\label{thm:main_recall}
    Let $\C\subset \mathbb{P}^2_k$ be an irreducible curve of degree $n$ defined by an homogeneous polynomial $F$ monic and separable in $X_2$.  Let $D\in Div(\C)$. There exists a deterministic algorithm which computes a compressed basis (Definition \ref{def:compressed_basis})  of $\mathcal{L}(D)$ with $$\Tilde{\O}(n^\omega(\deg(D^+)+\deg(D^-)+\delta(\C)+n ))\subset \Ot(n^\omega(\deg(D^+)+\deg(D^-)+n^2))$$
    arithmetic operations in $k$. If $\deg D\ge 0$, the complexity becomes 
    $$\Tilde{\O}(n^\omega(\deg(D^+)+\delta(\C)+n)).
    $$
    If moreover the singular locus of $\C$ and the support of $D$ are contained in the affine chart $X_0\ne 0$, the complexity becomes 
    $$\Tilde{\O}(n^\omega(\deg(D^+)+\delta(\C))).$$
\end{thm}

\begin{proof}
We have $\lambda(f)=1$ in this setting. Hence, the first result follows  straightforwardly from Theorem \ref{thm:RR} combined with Proposition \ref{prop:bound for expI} and Corollary \ref{cor:delta_vs_index}. The second statement (that is Theorem \ref{thm:main}) follows immediately since $\deg D\ge 0$ implies $\deg D^-+\deg D^+\le 2\deg D^+$. The last estimate follows from Theorem \ref{thm:complexity trivial index}.
\end{proof}

We may want to store all elements of a $k$-basis of $\L(D)$. We get the following size estimate.

\begin{prop}\label{prop:expanded_basis}
Suppose $\C$ irreducible over $\bar{k}$. Given a compressed basis of $\L(D)$, we can compute a $k$-basis of $\L(D)$ in time 
$$
\O(n(\deg(D)+1) (\deg(D^+) +\delta(\C)+n))
$$
(that is zero cost if $\deg D<0$). 
\end{prop}

\begin{proof}
The complexity is that of writting all elements of the set $\{t^j b_i, \, 0\le j \le d_i, 0,i\le n-1\}$. This set has cardinality $\ell(D):=\dim \L(D)$. If $\ell(D)=0$, there is nothing to do. Otherwise, $\deg D\ge 0$, hence $\deg D^+ + \deg D^-\le 2\deg D^+$.  Thus each element $t^j b_i\in\L(D)$ has arithmetic size $n(\deg D^+ +\delta(\C) +n)$ by (the proof) of Theorem \ref{thm:RR} and Corollary \ref{cor:delta_vs_index}. The assumption $\C$ irreducible over $\bar{k}$ is equivalent to $k$ is algebraically closed in $L$. In such a case, it is well known that $\ell(D)\le \deg(D)+1$ (see e.g. \cite[Prop.1.4.9 and Eq.1.21]{stichtenoth}), which concludes the proof. 
\end{proof}

\begin{rmq}
If  $\C$ is only irreducible over $k$, its number of irreducible factors over $\bar{k}$ is $\rho=[k_0:k]=\dim \L(0)$ where $k\subset k_0\subset L$ is the field of constants, algebraic closure of $k$ in $L$. In such a case, we have the inequality  $\ell(D)\le \deg D+ \rho$ by (the proof of) \cite[Prop.1.4.9]{stichtenoth} and we have to replace the factor $\deg D+ 1$ by $\deg D + \rho$ in Proposition \ref{prop:expanded_basis}. For instance, the curve defined by $f=X^2-2 t^2$ is irreducible over $\Q$, but has $\rho=2$ irreducible components over $\bar{\Q}$. The space of constant functions is $\L(0)=k_0\simeq k(\sqrt{2})$, with $k$-basis $(1,x/t)$ and dimension $\ell(0)=2$.
\end{rmq}

\subsection{Reduction to the monic separable case}\label{sec:monic_separable}

For the sake of completeness, let us explain how one can reach the monic and separable assumption of Theorem \ref{thm:main_recall} given an arbitrary projective plane curve. 

\begin{lem}\label{lem:non_vanish}
Let $F\in k[X_0,X_1,X_2]$ irreducible, homogeneous of degree $n>0$. Suppose that $\Card(k)>n$. We can compute two distinct rational points $P,P'\in \pr^2(k)$ such that $F(P)\ne 0$ and $F(P')\ne 0$ in time $\O(n^2)$. 
\end{lem}

\begin{proof}
Let $f(t,X)=F(1,t,X)$. We can write $f=\sum f_i(t) X^i$ with $\deg f_i\le n-i$. Evaluating $f_i$ at $t=0$ and $t=1$ with Horner's rule costs $\O(n)$ operations. Summing over $i$, we compute $f(0,X)$ and $f(1,X)$ in time $\O(n^2)$. Let $\alpha_0,\ldots,\alpha_n$ be distinct elements of $k$. Evaluating $f(0,X)$ and $f(1,X)$ at $\alpha_0,\ldots,\alpha_n$ costs $\Ot(n)$ operations using fast multi-point evaluation. Since each polynomial has degree at most $n$, there exists $\alpha=\alpha_i$ and $\alpha'=\alpha_j$ such that $f(0,\alpha)\ne 0$ and $f(1,\alpha')\ne 0$. The points $P=(1:0:\alpha)$ and $P'=(1:1:\alpha')$ satisfy the required assumptions.
\end{proof}

\begin{lem}\label{lem:Ftilde}
Let $P,P'$ as above. There exists a projective automorphism $\tau:\pr^2_k\to \pr^2_k$ which sends $(0,1,0)$ to $P$ and $(0,0,1)$ to $P'$. We can compute $\tilde{F}=F\circ \tau$ in time $\Ot(n^2)$. The polynomial $\tilde{F}$ is homogeneous of degree $n$, irreducible, and monic (up to multiplication by a non zero scalar) of degree $n$ both with respect to $X_1$ and $X_2$ 
\end{lem}

\begin{proof}
First claim is clear : this amounts to find an invertible linear map from $\A^3\to \A^3$ which maps two given non zero vectors to two given non zero vectors. Computing $\tau$ is negligible and the second claim is proved for instance in \cite[Lem.2.5]{ACL2}. Of course, $\tilde{F}$ remains homogeneous of degree $n$ and irreducible. By construction, we have $\tilde{F}(0,1,0)=F(P)\ne 0$ and $\tilde{F}(0,0,1)=F(P')\ne 0$. This exactly means that the monomials $X_1^n$ and $X_2^n$ appear in $\tilde{F}$ with a non zero coefficient. 
\end{proof}

\begin{lem}\label{lem:separable}
Let $F\in k[X_0,X_1,X_2]$ be homogeneous of degree $n>0$, irreducible, and monic of degree $n$ in $X_1$ and $X_2$. If $k$ is perfect, then $F$ is separable with respect to $X_1$ or separable with respect to $X_2$.
\end{lem}

\begin{proof}
Suppose on the contrary that $F$ is not separable w.r.t $X_1$ and $X_2$. Since $F$ is irreducible, we thus have $\partial_{X_1} F=\partial_{X_2} F=0$. Since $F$ is monic of degree $n$ in $X_1$, this forces $p=\car(k)$ to divide $n$ and the Euler formula $\sum X_i\partial_{X_i} F= nF=0$ forces $\partial_{X_0} F=0$. Finally, we get $F=G(X_0^p,X_1^p,X_2^p)$ for some $G$. Since $k$ is perfect, this leads to $F=G^p$, contradicting that $F$ is irreducible.
\end{proof}

We are led to the following result.

\begin{prop}\label{prop:monic_separable}
Let $F\in k[X_0,X_1,X_2]$ irreducible, homogeneous of degree $n>0$. If $\Card(k)>n$, we can compute in time $\Ot(n^2)$ a polynomial $f\in k[t,X]$ of total degree $n$, separable and monic of degree $n$ with respect to $X$ such that 
$k(\C)\simeq k(t)[X]/(f)$. 
\end{prop}

\begin{proof}
We compute $P,P'$ as in Lemma \ref{lem:non_vanish} and consider $\tilde{F}$ as in Lemma \ref{lem:Ftilde}. Following Lemma \ref{lem:separable}, we let $f(t,X)=\tilde{F}(1,t,X)$ if $\tilde{F}$ is separable with respect to $X_2$ or $f=\tilde{F}(1,X,t)$ if $\tilde{F}$ is separable with respect to $X_1$. 
\end{proof}

If $\Card(k)<n$,  we may have $F(P)=0$ for all $P\in \pr^2_k$ (filling curves), in which case Lemma \ref{lem:non_vanish} can consider a finite extension $k'$ of $k$ such that $\Card(k')>n$ and $[k':k]\in O(\log(n))$. Since $F$ is irreducible over $k$, its irreducible factors over $k'$ are conjugate over $k$ and it's straightforward to check that Lemma \ref{lem:separable} still holds in this context. Thus we are led to compute a basis of 
$$\L(D)\otimes k'\simeq \L(D')$$
where $D'$ is the divisor of the curve $\C'=\C\otimes k'$ defined by a degree $n$ homogeneous polynomial $G\in k'[X_0,X_1,X_2]$ which is separable and monic of degree $n$ in $X_2$. The curve $\C'$ is not necessarily irreducible over $k'$ but the all algorithm works exactly in the same way. The cost estimate is multiplied at most by a logarithmic factor $O(\log(n))$. In order to recover a basis of $\L(D)$ from a basis of $\L(D)\otimes k'$ we can use for instance the relative trace map
$$
\Tr_{k'/k}:\L(D)\otimes k' \to \L(D).
$$
A choice of a random basis of $\L(D)\otimes k'$ will lead to a $k$-basis of $\L(D)$ with probability $\ge 1/4$. We refer to \cite[Sec.7.5]{ACL2} for details (see also \cite{Huang-Ierardi} for a deterministic approach).

\begin{rmq}\label{rem:monic_separable_assumption}
Since $A=k[t]$ is a $k$-algebra of finite type, the fractional ideals of $B$ are always free $A$-module of rank $n=[L:K]$ by \cite[Ch.V,§2,no2, thm.2]{Bourbaki}. This suggests that the hypothesis monic and separable can probably be dropped. This would require to adapt the OM algorithm and the computation of integral bases in this larger context, which is beyond the scope of this paper. This would be of great interest, since then our algorithm would work for any irreducible polynomial $F$, avoiding any change of coordinates and any base field extension. 
\end{rmq}

\appendix

\section{Valuation rings and integral closures}\label{sec:A}

A common ingredient of geometric and arithmetic approaches is the notion of integral closure. Let us briefly remind the main properties we need. 

\begin{definition}\label{def:integral_closure}
Let $S$ be a ring and $R\subset S$ a subring. An element $x\in S$ is integral over $R$ if it is the root of a \emph{monic} polynomial $f\in R[X]$. The set of integral elements of $S$ over $R$ is called the integral closure of $R$ in $S$. We say that $R$ is integrally closed in $S$ if it coincides with its integral closure in $S$. If $R$ is integrally closed in its total ring of fractions, we say simply that $R$ is integrally closed.
\end{definition}

\begin{prop}\cite[Cor 5.3 and Cor. 5.5]{Atiyah}
Let $S$ be a ring and let $R\subset S$ be a subring. The integral closure of $R$ in $S$ is a subring of $S$ that contains $R$ and that is integrally closed in $S$. 
\end{prop}

\medskip

For a ring $R$, we denote $\Spec\,R$ the set of maximal ideals of $R$. For $\p\in \Spec\,R$, we let $R_\p$ be the localization of $R$ at $\p$ (i.e. at the multiplicative subset $R\setminus\p$). It is a local ring with maximal ideal $\p R_\p$.

\begin{prop}\cite[Prop 5.13]{Atiyah}\label{prop:int_closed} An integral domain $R$ is integrally closed if and only if all local rings $R_\p$, $\p\in \Spec R$ are integrally closed. 
\end{prop}

Connection between integral closures and valuation rings starts with the following elementary result :

\begin{lem}\label{lem:valuation_ring_is_integrally_closed}
A valuation ring is integrally closed. 
\end{lem}

\begin{proof}
Let $\O$ be a valuation ring. Let $x\in \Frac(\O)$ such that $x^n=a_{n-1}x^{n-1}+\cdots +a_0$, where $a_i\in \O$.  If $x\notin \O$, then $x^{-1}\in \O$ by definition of a valuation ring. Dividing by $x^{n-1}$, we get $x=a_{n-1}+\cdots +a_0 x^{1-n}\in \O$, a contradiction. 
\end{proof}

The following result gives a characterization of the discrete valuation rings of a function field :

\begin{prop}\label{prop:DVR=integrally-closed}\cite[Prop.9.2]{Atiyah}
Let $\O$ be a Noetherian local ring of dimension one. Then $\O$ is a discrete valuation ring if and only if it is integrally closed.
\end{prop}

Integrally closed rings are of particular importance for our purpose, due to the following results (see e.g. \cite[Cor 5.22]{Atiyah}):

\begin{prop}\label{prop:integrally-closed=DVRintersection}
The integral closure of a subring $R$ of a field $L$ is the intersection of all valuation rings of $L$ containing $R$. 
\end{prop}

Let us assume now that $R$ is a subring of a function field $L/k$, with $k\subset R$. The previous proposition says that $b\in L$ is integral over $R$ if and only if $v(b)\ge 0$ for all valuations $v$ of $L$ which are positive on $R$.  

\begin{prop}\label{prop:Rbar_to_R} The integral closure $\overline{R}$ of $R$ in $L$ is a finite $R$-module and the map 
$$
\Spec \overline{R}\to \Spec R,\quad \p\mapsto \p\cap R$$
is well-defined, surjective, and finite. Moreover  this map is one-to-one above (at least) all $\m \in \Spec R$ such that $R_\m$ is a valuation ring (i.e. integrally closed). 
\end{prop}

\begin{proof}
If $R$ is a field, then $\overline{R}$ is a field which is a finite extension of $R$ and the proposition is obvious in such a case. If $R$ is not a field, then $R$ is a finitely generated $k$-algebra of Krull dimension one, and $\overline{R}$ is again a finite $R$-module by the finiteness theorem (see e.g. \cite[Thm.9, p.267]{samuel}). The claims then follow from \cite[Lem.F.9 and Thm.F.10]{kunz}
\end{proof}
\begin{definition}
The \emph{center} of a place $P\in V_L$ on $R$ is the set $\p=P\cap R$, that is the set of "functions" $b\in R$ which vanish at $P$.
\end{definition}

The center of a place is a prime ideal of $R$. 
In what follows, we say that a subring $S\subset L$ lies over $R$ if $R\subset S$.

\begin{prop}\label{prop:center}
Suppose that $R$ is not a field. Let $\overline{R}$ be the integral closure of $R$ in $L$ and let $P\in V_L$. The following conditions are equivalent :

(a) The center $\p=P\cap R$ is not the zero ideal;

(b) The center $\p=P\cap R$ is a maximal ideal of $R$;

(c) The valuation ring $\O_P$ lies over $R$;

(d) The valuation ring $\O_P$ lies over $\overline{R}$.

\noindent 
We say then that $P$ is centered on $R$ and that $P$ lies over $\p$.
\end{prop}

\begin{proof}
$a)\Leftrightarrow b)$ follows from the fact that $R$ has Krull dimension one since $k\subset R\subset L$ is not a field.   
$d)\Rightarrow c)$ is clear since $R\subset \overline{R}$. 
Let us show $c)\Rightarrow d)$. If $R\subset \O_P$, then any element of $L$ which is integral over $R$ is a fortiory integral over $\O_P$, hence $\overline{R}\subset \overline{\O_P}$. But $\overline{\O_P}=\O_P$ by Lemma \ref{lem:valuation_ring_is_integrally_closed}. 
We recall that $\O_P=\{ a\in L \mid a^{-1}\notin P\}$. Let us prove $c)\Rightarrow a)$. Let $R\subset \O_P$. If $P\cap R=\{0\}$, then $a^{-1}\in \O_P$ for all nonzero $a\in R$ which leads to $L\subset \O_P$. 
It is not possible, so $P\cap R$ is not the zero ideal.
At last, we show $b)\Rightarrow c)$. Suppose that there exists $a \in R$ such that $a\notin \O_P$. In particular, $a\notin \p$. 
Since $\p$ is a maximal ideal of $R$ there exists $b\in R$ satisfying $ab-1\in \p$. However, $a \notin \O_P$ so $b\notin \O_P^{\times}$, hence $b\in \p$ follows. This leads to $1\in \p$, which is absurd, so no such element $a$ exist.
\end{proof}


\begin{prop}\label{prop:Chevalley} 
Let $V_L(R)$ be the set of places centered on $R$. Suppose that $\Frac(R)=L$. The map 
$$
V_L(R)\to \Spec \overline{R},\quad P\mapsto \p=P\cap \overline{R}
$$
is bijective, and we have $\O_P=\overline{R}_\p$ and $P=\p \overline{R}_\p$.
\end{prop}

\begin{proof}
We have $V_L(R)=V_L(\overline{R})$ by items (c) and (d) of Proposition \ref{prop:center}, hence we may suppose that $R=\overline{R}$. The result follows from \cite[Prop.3.2.9]{stichtenoth}. 
\end{proof}



\section{Geometric point of view : places vs closed points on curves}\label{sec:curves}

\emph{Warning : we denote here by $P$ or $Q$ closed points of curves, which are not to be confused with places of function fields.}

\subsection{Affine curves and closed points}\label{ssec:closed_points}
Let $k$ be a perfect field and  $\bar{k}$ be an algebraic closure. Let $R=k[X_1,\ldots,X_n]$ be the polynomial ring in $n$ variables and let us consider the affine $n$-space  $$\A^n_k=\Spec\, R.$$
A closed point $P\in \A_n^k$ is by definition a maximal ideal of $R$.  Alternatively, we can think $P$ as the set of points in $\bar{k}^n$ at which all functions in $P$ vanish  : this set is a Galois orbit of a point in $\bar{k}^n$ under $\Gal(\bar{k}/k)$. Conversely, given a Galois orbit in $\bar{k}^n$, the set of all polynomials in $R$ vanishing at any (thus all) points of the orbit is a maximal ideal of $R$ (Nullstellensatz). An affine curve over $k$ is a one-dimensional integral subscheme $C\subset \A_n^k$, that is
$$
C=\Spec \, R/I
$$ 
where $I\subset R$ is a prime ideal of height $n-1$. 
A closed point of $C$ is now a maximal ideal of $R/I$, or equivalently a maximal ideal of $R$ containing $I$ by the correspondence theorem. Again, a closed point  $P\in C$ can be thought as a Galois orbit of a point of the variety (the geometric curve)
$$
C(\bar{k})=\{a\in \bar{k}^n,\,\, f(a)=0\,\, \forall\,\, f\in I\}.
$$
The ring of regular functions on $C$, or coordinate ring of $C$, is
$$
k[C]:=R/I.
$$
As $I$ is assumed to be prime, $k[C]$ is an integral domain. We call its  field of fractions the field of rational functions of $C$, denoted by $k(C)$. It is a function field and every function field is isomorphic to a field of shape $k(C)$. The local ring of $C$ at a closed point $P\in C$ is 
the localization $k[C]_P$ of $k[C]$ at the maximal ideal $P$, usually denoted $\O_{C,P}$. We thus have inclusions
$$
k[C]\subset \O_{C,P}\subset k(C).
$$ 
The field $k_P=\O_{C,P}/P \O_{C,P}\simeq k[C]/(P) $ is the residue field of $P$, which is a finite extension of $k$.  If $h\in \O_{C,P}$, we say that $h$ is regular at $P$. This terminology is justified by the fact that the evaluation map
$$
ev_P:\O_{C,P} \to k_P,\quad ev_P(h)=h \mod P
$$
is well defined. We write also $h(P)=ev_P(h)$. We say that $h$ vanishes at $P$ if  $h(P)=0$, that is if $h\in P$. Note that if $k=\bar{k}$ then $P=(X_1-a_1,\ldots,X_n-a_n)$ can be identified with  $a=(a_1,\ldots,a_n)\in C(\bar{k})$ and $h(P)$ is canonically identified with $h(a)$, usual evaluation of a function at a point. 

\begin{rmq}
The scheme $C$ is endowed with a topology : the closed subsets of $C$ are $C$ itself or finite unions of closed points. It comes also with a structural sheaf $\mathcal{O}_C$ on the topological space $C$, where for an open set $U\subset C$,  $\O_C(U)\subset k(C)$ is the subring of functions regular at all $P\in U$. The stalk at $P$ is the local ring $\O_{C,P}$. These are crucial data, especially to define morphisms between schemes. However, we deliberately don't go into these details here, and we mainly consider poorly $C$ as a set. See any book of algebraic geometry for more details.
\end{rmq}

\subsection{Regular and singular points.} The local rings $\O_{C,P}$ are not necessarily valuation rings. Indeed, we have the following classical characterization \cite[Prop.9.2, p.94]{Atiyah}:

\begin{prop}\label{prop:regular_point} Let $P\in C$ be a closed point. The following assertions are equivalent :

(a) The ring $\O_{C,P}$ is a discrete valuation ring.

(b) The ring $\O_{C,P}$ is integrally closed.

(c) If $I=(f_1,\ldots,f_r)$, the partial derivatives of the $f_i$'s do not all vanish at $P$. 
\end{prop}

\noindent
We say that $P$ is a \emph{regular point} of $C$ if it satisfies one of these conditions. Otherwise, we say that $P$ is a singular point. By (c), $P$ is a regular point if and only if the geometric curve $C(\bar{k})$ is smooth (in the usual sense) at all points of the Galois orbit defined by $P$.  
A closed point is singular if it's not regular. A curve $C$ is regular (or smooth) if it has no singular points. Otherwise, we say that $C$ is singular. The set of singular points on a curve $C$ is finite. By combining Proposition \ref{prop:int_closed} and Proposition \ref{prop:regular_point}, we get the classical result:

\begin{prop}\label{prop:regular_curves}
An affine curve $C$ is smooth if and only if  $k[C]$ is integrally closed. 
\end{prop}

Note that a regular point  $P\in C$ defines in particular a discrete rank one valuation $v_P:k(C)\to \Z\cup\{\infty\}$. Thus it makes sense to talk about the vanishing order or the pole order of a rational function $h\in k(C)$ at a regular point.

\subsection{Normalization of singular curves.} \emph{The local ring of a curve at a singular point is not a valuation ring.} In such a case, the notion of vanishing order or pole order of a function at $P$ does not make sense anymore. We need to desingularize, or normalize, the curve $C$. 

\begin{definition}
The normalization of $C$ is the curve 
$\Cbar=\Spec\, \kCbar$
where $ \kCbar$ stands for the integral closure of $k[C]$ in $k(C)$.
\end{definition}

\noindent
The fact that $\Cbar$ is a curve follows from the finiteness theorem (which is itself a consequence of Noether's normalization theorem) which ensure that $ \kCbar$ is a one-dimensional integral $k$-algebra (Proposition \ref{prop:Rbar_to_R}). That is,
$\kCbar\simeq k[Z_1,\ldots,Z_m]/J$
where $J$ is a prime ideal of height $m-1$. Thus $\overline{C}$ is indeed an affine curve in a (possibly larger) affine space $\A^m_k$. 

\begin{ex}\label{ex:blow-up}
Consider the affine plane curve $C=\Spec k[T,X]/( X^2-T^2-T^3)$. There is only one singular point $P=(T,X)$ (corresponding to the geometric point $(0,0)$). Denote $t$ and $x$ the residue classes of $T$ and $X$. Thus $k[C]=k[t,x]$ where $x^2=t^2-t^3$. The rational function $y=x/t\in k(C)$ satisfies $y^2-t-1=0$, thus is integral over $k[C]$. Adjoining this element to the ring $k[C]$ leads to the bigger ring
$$
k[t,x,x/t]= \frac{k[T,X,Y]}{(X^2-T^2-T^3,X-TY,Y^2-T-1)}\simeq k[Y].
$$
As $k[Y]$ is integrally closed, we deduce that $\kCbar=k[t,x,x/t]\simeq k[Y]$ and 
$\Cbar$ is isomorphic to the affine line $\A^1_k=\Spec k[Y]$ (note however that $\Cbar$ is naturally embedded in $\A_k^3$ here). 
\end{ex}

By Proposition \ref{prop:regular_curves}, the affine curve $\Cbar$ is smooth, and $\Cbar=C$ if and only if $C$ is smooth. We have inclusions
$k[C]\subset \kCbar \subset k(C)$
so the curves  $C$ and $\Cbar$ have same function field $k(C)=k(\Cbar)$. Two curves with isomorphic function fields are said to be birationally equivalent.
By Proposition \ref{prop:Rbar_to_R}, we have a surjective map (in fact a morphism of $k$-schemes)
$$
\pi:\Cbar \twoheadrightarrow C,\quad \pi(Q)=Q\cap k[C]
$$ 
which is finite, and one-to-one except above a finite number of closed points of $C$, these points being contained in the singular locus of $C$. If $P\in C$ is singular, there may be several points or not lying above $P$. 

\medskip

\begin{ex} (Example \ref{ex:blow-up} continued)
The  fiber of $\pi:\Cbar \to C$ above $P=(T,X)$ consists of the maximal ideals $Q\subset k[T,X,Y]$ containing $P+(X^2-T^2-T^3,X-TY,Y^2-T-1)=(T,X,Y^2-1)$. Assuming $\car(k)\ne 2$, the fiber thus consists of two closed points, given by the maximal ideal $Q_1=(T,X,Y-1)$ and $Q_2=(T,X,Y+1)$, corresponding to the geometric points $(0,0,\pm 1)$ (this is an example of a blow-up of a plane curve at a nodal point). If $\car(k)= 2$ then $Y^2-1=(Y-1)^2$ and there is a unique point $Q$ lying above $P$, given by the maximal ideal $(T,X,Y-1)$. 
\end{ex}

\medskip

We say that a place of $k(C)$ is centered on $C$ if it is centered on $k[C]$ (Proposition \ref{prop:center}). We can reformulate Proposition \ref{prop:Chevalley} in this context : 

\begin{thm}\label{thm:affine_normalisation}
Let $C$ be an affine curve. There is a one-to-one correspondence between the places of $k(C)$ centered on $C$ and the closed points of $\Cbar$.
\end{thm} 

The underlying map is $\tilde{Q}\mapsto Q=\tilde{Q}\cap k[\Cbar]$. The center on $C$ of the place $\tilde{Q}$ is  $P=\tilde{Q}\cap k[C]=\pi(Q)\in C$. If the curve $C$ is smooth, we may thus identify places centered on $C$ and closed points of $C$.  

\subsection{Places and branches of plane curves.}\label{ssec:plane_branches}
Suppose that $C=\Spec k[T,X]/(f)$ is a plane curve defined by an irreducible polynomial $f\in R=k[T,X]$ and let $P\in C$ be a closed point. We can describe  the places centered at $P$ as follows.  Let $\widehat{R}_{P}$ be the completion of $R$ with respect to the $P$-adic topology. Typically, $P=(T,X)$ and $\widehat{R}_{P}=k[[T,X]]$. The germ of the curve $C$ at $P$ is
$$(C,P):=\Spec \widehat{R}_{P}/(f).$$
The local irreducible factorization $f=f_1\cdots f_r\in \widehat{R}_{P}$ leads to the decomposition 
$$
\widehat{R}_{P}/(f)= \widehat{R}_{P}/(f_1)\oplus \cdots \oplus \widehat{R}_{P}/(f_r)
$$
where each summand is now an integral local ring (but not integrally closed in general). The irreducible germ of curve $(C_i,P)$ defined by $f_i$ is called \emph{a branch} of $(C,P)$. We have
$$
(C,P)=(C_1,P)\cup \cdots \cup (C_r,P).
$$
The following result is the local counterpart of Proposition \ref{prop:primes_vs_local_factors} and Proposition \ref{prop:valuation_via_resultant}.
\begin{prop}\label{prop:places_vs_branches}
There is a one-to-one correspondence between the places $Q_i$ of $k(C)$ with center $P$ and the branches $(C_i,P)$ of the germ of curve $(C,P)$.  The valuation $v_{Q_i}$ is induced by the intersection multiplicity at $P$ with $f_i$. More precisely, if $h \in\mathcal{O}_{C,P}$, we have the equality 
$$
v_{Q_i}(h)=\frac{1}{[k_P:k]}\dim_{k} \frac{\widehat{R}_{P}}{(f_i,h)}.
$$
where we still denote $h\in R_P\subset \widehat{R}_{P}$ an arbitrary lifting of $h$.
\end{prop}

This point of view allows to represent a place $Q_i$ by the (truncated) local equation $f_i$ of the corresponding branch, and the intersection multiplicity can be computed by means of resultants or local parametrizations, as detailled in Section \ref{ssec:parametrization}.

\subsection{Places at infinity and compactification.} 
Unfortunately, the closed points of $\Cbar$ do not allow to describe all places of $k(C)$. In other words, there are some remaining places which are not centered on $C$. These are the places "at infinity" of the affine curve $C$. 

\begin{ex}\label{ex:places_infinity}
Consider the regular curve $C=\Spec k[T,X]/(TX-1)$, with coordinate ring $k[C]=k[t,x]$, where $tx=1$. Thus $k[C]=k[t,1/t]$ and $k(C)=k(t)$. Consider the two places $P_1$ and $P_2$ given respectively by the maximal ideals $(t)$ and $(1/t)$ of the discrete valuation rings $k[t]_{(t)}$ and $k[1/t]_{(1/t)}$ of $k(C)$. We check that $P_1\cap k[C]=\{0\}$ and $P_2\cap k[C]=\{0\}$, thus $P_1$ and $P_2$ are not centered on $C$ (Proposition \ref{prop:center}). 
\end{ex}

In order to give a geometric meaning of the places at infinity, we need to "compactify" $C$. 
A classical compactification is the projective closure. For the sake of completeness, let us briefly explain the construction. The  projective $n$-space over $k$ is
$$
\pr^n_k=\Proj\,k[X_0,\ldots,X_n].
$$ 
Set theoretically, $\pr^n_k$ is the set of homogeneous prime ideals of $R'=k[X_0,X_1,\ldots,X_n]$ of height $n$ whose radical is not equal to the irrelevant ideal $(X_0,\ldots,X_n)$. The underlying variety is the usual projective $n$-space over $\bar{k}$, that is
$$\pr^n(\bar{k})= (\bar{k}^{n+1}\setminus  0)/\sim
$$
where $\sim$ is the collinearity relation of equivalence. The hyperplane at infinity of $\pr^n_k$ is
$$
H_0=\Proj \, R'/(X_0)\simeq \Proj \,k[X_1,\ldots,X_n] \simeq \pr^{n-1}_k,
$$ 
denoted for short $H_0=\{X_0=0\}$. The closed points of $H_0$ corresponds to the closed points of $\pr^n_k$ containing $(X_0)$. The remaining points of $\pr^n_k$ are one-to-one with the maximal ideal of $R=k[X_1,\ldots,X_n]$ by dehomogeneization, that is, by letting $X_0=1$. We get in such a way a bijection
$$
\iota_0: U_0=\pr^n_k\setminus H_0 \longrightarrow \A^n_k. 
$$ 
The reverse map $\iota_0^{-1}$ consists of homogeneization of maximal ideals $P\subset k[X_1,\ldots,X_n]$ with respect to $X_0$, a generator $f$ of $P$ being sent to the homogeneous polynomial 
$$
f(X_1,\ldots,X_n)\mapsto F(X_0,\ldots,X_n)=X_0^{d} f(X_1/X_0,\ldots,X_n/X_0)
$$
where $d=\deg(f)=\deg(F)$. Note that for each $i=0,\ldots,n$, dehomogeneization with respect to $X_i$ leads to an isomorphism  $\iota_i:U_i=\pr^n_k\setminus H_i\simeq \A^k_n$ where $H_i=\{X_i=0\}$. Moreover, the charts $U_i$ cover the all space $\pr^n_k$, thus we can work locally in $\pr^n_k$ as in the usual affine $n$-space. 

\medskip

Let $I'\subset R'$ be the homogeneization of the ideal $I\subset R$ defining the affine curve $C$. The closed points of
$$\C=\Proj\, R'/I'$$
are the homogeneous prime ideals of $R'$ of height $n$ containing $I'$ and whose radical is not equal to $(X_0,\ldots,X_n)$.  We have a natural inclusion $\C\subset \pr^n_k$ and the dehomogenization map $\iota_0$ induces an isomorphism
$$
\iota_0 : \C\cap U_0 \longrightarrow C,
$$
hence an isomorphism $\O_{\C,P}\simeq \O_{C,\iota_0(P)}$ for all $P\in \C\cap U_0$. The remaining set of points of $\C$ is $H_0\cap \C$ (the points "at infinity"), and there are finitely many such points. We say that $\C$ is \emph{the projective closure} of $C$. For each chart $U_i$, the ideal of the affine curve $C_i=\iota_i(\C\cap U_i)$ is easily deduced from the ideal of $C$.
In the plane case $C=\Spec k[X_1,X_2]/(f)$ with $f$ of total degree $d$, we have 
$$
\C=\Proj \, k[X_0,X_1,X_2] /(F)
$$
where $F$ is the homogeneization of $f$ and we get $C_1=\Spec k[X_0,X_2]/(f_1)$ and $C_2=\Spec k[X_0,X_1]/(f_2)$ where 
$$
f_1(X_0,X_2)=X_0^d f(1/X_0,X_2/X_0)\quad \text{and}\quad f_2(X_0,X_1)=X_0^d f(1/X_0,X_1/X_0).
$$
The curve $\C$ is naturally isomorphic to the abstract curve obtained by gluing the affine curves $C_i$ along their intersection. Each affine curve $C_i\subset \A^n_k$ admits a normalization $\Cbar_i$ as described in the previous paragraph. The affine curves $\Cbar_i$ can be glued again together and the resulting curve $\overline{\C}$ is a \emph{smooth projective curve}, called the normalization of $\C$.

\begin{thm}\label{thm:projective_normalisation}
Let $C\subset \A^n_k$ be an affine curve with projective closure $\C\subset \pr^n_k$. There is a one-to-one correspondence between  the places of $k(C)$ and the closed points of the normalization $\overline{\C}$ of $\C$. 
\end{thm}

The places of $k(C)$ which are centered on $C$ are called the affine places. They correspond to closed points of the normalization $\Cbar$ of $C$. The remaining places are centered on the points at infinity of $\C$ that is on $\C\setminus C=\C\cap (X_0=0)$. Take care that the normalization of the projective closure $\overline{\C}$ is usually not the projective closure of the affine normalization $\Cbar$. This is the case if and only if the points at infinity are regular points of $\C$.

\begin{ex}\label{ex:places_infinity_continued} (Example \ref{ex:places_infinity} continued). 
The homogeneization of $f=TX-1$ with respect to a new variable $Z$ is $F=TX-Z^2$, and  $\C=\Proj \, k[Z,T,X]/(F)$. The intersection of $\C$ with  the line at infinity $Z=0$ is $\Proj \, k[Z,T,X]/(F,Z)\simeq \Proj \, k[T,X]/(TX)$, which consists of two closed points $P_1=(Z,T)$ and $P_2=(Z,X)$ (corresponding respectively to  the geometric points $(0:0:1)$ and $(0:1:0)$). These two points correspond to the two places not centered on $C$, as described in Example \ref{ex:places_infinity}. 
\end{ex}

\begin{rmq}
There are various ways of compactifying an affine curve. For instance, we can consider its closure in a product of projective spaces or in proper toric varieties. For instance, considering a plane curve defined by a polynomial $f$ with $\lambda(f)>1$ suggests to consider the Zariski closure $\C$ of the affine plane curve of $f$ in some well-chosen \textit{weighted projective compactification}  of $\A^2_k$. 
\end{rmq}

\paragraph{Affine places versus finite places.} Let $C=\Spec k[T,X]/(f)$ be an affine plane curve and let $\C\subset \pr^2_k$ be the projective compactification of $C$. Denoting $K=k(T)$ and $L=K[X]/(f)$ as in Section \ref{sec:3}, we may identify $L=k(\C)$. In this context, affine places of $L/k$ (i.e. centered on $C$) and finite places of $L/K$ (Definition \ref{def:finite_places}) do not necessarily coincide. Indeed :

\begin{prop} \label{prop : finite/centered places}
The set of places centered on $C$ is contained in the set of finite places of $L/K$. 
These two sets coincide if and only if $f$ is monic in $X$.  
\end{prop}

\begin{proof}
By Proposition \ref{prop:center} (d), a place $P$ is centered on $C$ if and only if $\O_P$ contains $\kCbar$, while $P$ is finite if and only if $\O_P$ contains the integral closure $B$ of $A=k[t]$ in $L=k(t,x)$. Since $A\subset k[C]$, we have $B\subset \kCbar$, thus any place centered on $C$ is a finite place. If $f$ is monic, then $x\in B$, hence $B=\kCbar=\overline{A[x]}$ and affine places coincide with finite places. If $f$ is not monic, then $x\notin B$ and $B\varsubsetneq \kCbar=\overline{A[x]}$. Since both rings are integrally closed, we deduce from Proposition \ref{prop:integrally-closed=DVRintersection} that there exists a place $P$ such that $\O_P$ contains $B$ and does not contain $\kCbar$, that is $P$ is finite, but not centered on $C$. 
\end{proof}

For instance, in Exemple \ref{ex:places_infinity_continued}, the place with center $P_1=(Z,T)$ is not an affine place, but is a finite place.

\subsection{Delta invariant and the genus formula}\label{ssec:genus_formula}

Let $\C$ be a projective curve defined over $k$. Let $P\in \C$ be a closed point, with local ring $\O_{\C,P}$. The normalization $\overline{\O_{\C,P}}$ is a semi-local ring, whose maximal ideals correspond to the closed points of the normalization $\overline{\C}$ lying over $P$. It is a finite $\O_{\C,P}$-module (Proposition \ref{prop:Rbar_to_R}), thus $\overline{\O_{\C,P}}/\O_{\C,P}$ is a finite $k$-dimensional vector space. 

\begin{definition}\label{def:delta-invariant}
The delta-invariant of $\C$ at a closed point $P\in \C$ is 
$$
\delta_P(\C):=\dim_k \overline{\O_{\C,P}}/\O_{\C,P}.
$$ 
The global delta-invariant of $\C$ is $\delta(\C)=\sum_{P\in \C} \delta_P(\C)$.
\end{definition}

Note that $\delta_P(\C)= 0$ if and only if $P$ is a regular point (Proposition \ref{prop:regular_point}). Thus the sum holds over the finite set of singular points and the global delta-invariant is well-defined. It measures how singular the curve $\C$ is. If $\C$ is a nodal curve, its delta-invariant is simply the number of nodes (over $\bar{k}$), and is called sometimes the "double points number". 

\medskip
\noindent
Denoting $\O_{\C}$ the structural sheaf of $\C$, the arithmetic genus of $\C$ is defined as
$$
p_a(\C)=1-\dim_k H^0(\C,\O_{\C})+\dim_k H^1(\C,\O_{\C}).
$$
If $\C$ is irreducible over $\bar{k}$, this is $p_a(\C)=\dim_k H^1(\C,\O_{\C})$. 
By definition, the geometric genus $g(\C)$ is the arithmetic genus of the normalization of $\C$. In contrast to the arithmetic genus which depends on the embedding of $\C$, the geometric genus is a birational invariant.

The next result due to Hironaka \cite{Hironaka} is known as the genus formula. It relates the delta invariant, the arithmetic genus, and the geometric genus of a projective curve.

\begin{prop}\cite[Thm 2]{Hironaka}\label{prop:Hironaka}
Let $\C\subset \pr^r_k$ be a projective curve. We have $g(\C)=p_a(\C)-\delta(\C)$. 
\end{prop}

In other words, the difference between arithmetic genus and geometric genus measures how singular the curve is. For a projective plane curve of degree $n$, the genus degree formula states that 
$$p_a(\C)=\frac{(n-1)(n-2)}{2}.
$$ 
Combined with Proposition \ref{prop:Hironaka}, we get:
 
\begin{coro} \label{coro:lien genre delta}
Let $\C\subset \pr^2_k$ be a projective plane curve of degree $n$. Then $g(\C)=\frac{(n-1)(n-2)}{2}-\delta(\C)$. 
\end{coro}

%
%
%
%
%

\bibliographystyle{abbrv}
\bibliography{biblio}
\addcontentsline{toc}{section}{References.}

\end{document}